\documentclass[final]{siamltex}

\newtheorem{remark}{Remark}[section]

\usepackage{times,amsmath,amsbsy,amssymb}

\usepackage{graphicx}

\title{RESIDUAL-BASED A POSTERIORI ERROR ESTIMATION FOR MULTIPOINT FLUX MIXED FINITE
ELEMENT METHODS\thanks{This work was supported in part by the Education Science Foundation of Chongqing
(KJ120420),
 National Natural Science Foundation of China (11171239), The Project-sponsored
by Scientific Research Foundation for the Returned Overseas Chinese Scholars and
 Open Fund of  Key Laboratory of Mountain Hazards and Earth Surface Processes, CAS.}}

\author{Shaohong Du\thanks{School of Science, Chongqing Jiaotong University,
                           Chongqing 400047, China, ({\it
                           dushhong@gmail.com}).}
                           \and Shuyu Sun\thanks{ Computational Transport Phenomena Laboratory, The Physical Sciences and Engineering Division, King Abdullah University of Science and Technology, Thuwal 23955-6900, Kingdom of Saudi Arabia, ({\it  shuyu.sun@kaust.edu.sa}).}
\and Xiaoping XIE\thanks{Corresponding author. School of Mathematics, Sichuan University,
Chengdu 610064, China ({\it xpxie@scu.edu.cn}). 
}}
\begin{document}
%\begin{flushleft}
%\hrulefill\\
%\end{flushleft}
\maketitle
\begin{small}
  {\bf{Abstract.} \rm{A novel residual-type {\it a posteriori } error analysis technique is
  developed for multipoint flux mixed finite element methods for flow in porous media
  in two or three space dimensions. The derived {\it a posteriori } error estimator for  the velocity and pressure error
  in $L^{2}-$norm  consists of discretization and quadrature indicators, and is shown to be
  reliable and efficient. The main tools of analysis are   a
  locally postprocessed approximation to the pressure solution of
  an auxiliary problem and a quadrature
  error estimate. Numerical experiments are presented to  illustrate
  the competitive behavior of the estimator.}}
\end{small}

\begin{it} Key words.\end{it} multipoint flux mixed finite element
method, postprocessed approximation, {\it a posteriori } error estimate

\begin{it} AMS subject classifications.\end{it} 65N06, 65N12, 65N15, 65N30, 76S05,

\pagestyle{myheadings} \thispagestyle{plain} \markboth{SHAOHONG DU \ \ \ SHUYU SUN \ \ \ XIAOPING XIE
}{A RESIDUAL-BASED POSTERIORI ERROR ESTIMATION FOR MFMFEM}

\section {Introduction}
Let $\Omega\subset{\mathbb{R}}^{d}$  be a bounded polygonal ($d=2$) or
polyhedral ($d=3$) domain with a Lipschitz continuous boundary
$\partial\Omega$.
We consider the following  first-order system of diffusion-type    partial differential equations:
\begin{equation}\label{equation1}
 \left \{ \begin{array}{ll}
    {\bf u}=-K\nabla p \quad   \mbox{in}\ \ \Omega,\\
      \nabla\cdot {\bf u}=f  \quad \mbox{in}\ \Omega,\\
      p=g  \quad \mbox{on}\ \Gamma_{D},\\
      {\bf u}\cdot{\bf n}=0  \quad \mbox{on}\ \Gamma_{N}.\\
 \end{array}\right.
\end{equation}
Here $\Gamma_{D},\ {\Gamma}_{N}$ are partitions of the boundary $\partial\Omega$ corresponding to the Dirichlet and Neumann conditions, respectively, with    $\partial\Omega=\bar{\Gamma}_{D}\cup\bar{\Gamma}_{N}$,
${\Gamma}_{D}\cap{\Gamma}_{N}=\emptyset $ and   $meas(\Gamma_{D})>0$, ${\bf n}$ is the outward unit normal vector on
$\partial\Omega$, and $K$ is a symmetric and uniformly positive
definite tensor with
\begin{equation}\label{newadd}
k_{0}\xi^{\rm T}\xi\leq\xi^{\rm T}K({\bf x})\xi\leq k_{1}\xi^{\rm
T}\xi,\ \ \forall\ {\bf x}\in\Omega,\ \forall\
\xi\in{\mathbb{R}}^{d}
\end{equation}
 for  $0<k_{0}\leq k_{1}<\infty$.
 This system has been widely used in physics to model diffusion processes such as heat or mass transfer and flow in porous media.
%We assume below that there is certain weak regularity on $K$ to be
 %satisfied (see (\ref{dual problem 1})).
 In flow in porous media, $p$ denotes the pressure, ${\bf u}$ is
the Darcy velocity, and $K$ represents the permeability divided by
the viscosity.
%The choice of boundary conditions is made for the
%sake of simplicity, since similar results are valid for other
%boundary conditions.

The main goal of this paper is to   derive
residual-based {\it a posteriori } error estimation for multipoint flux mixed finite element (MFMFE) methods for the model (\ref{equation1}).   The MFMFE approach was developed for single phase flow in porous media in \cite{Yotov10,Yotov06,Wheeler;Xue;Yotov2012}. It is motivated by the multipoint
 flux approximation (MPFA) approach \cite{Aavatsmark02,Aavatsmark98,Edwards, Klausen;Winther2006a, Klausen;Winther2006b}, which is a control volume method developed by the oil industry as a reliable discretization for   single-phase Darcy flow. One main advantage of  this method lies in that, by introducing sub-edge (or sub-face) fluxes, it  provides a local explicit flux with respect to the flow pressure, and allows for local flux elimination
 around grid vertices and reduction to a cell-centered pressure scheme.
The MFMFE
 method is  based on the lowest order Brezzi-Douglas-Marini (BDM1)   \cite{Brezzi;Fortin1985} or  Brezzi-Douglas-Duran-Fortin (BDDF1) \cite{Brezzi;Douglas;Duran;Fortin1987} finite element space. By using special quadrature rules, local velocity elimination is also attained which leads to   a symmetric and positive definite cell-centered system for the pressure on quadrilateral, simplicial and hexahedral meshes. In \cite{Wheeler;Xue;Yotov2014},  a coupling discretization of  MFMFE
 method and continuous Galerkin finite element method  was applied to  the poroelasticity system that describes fluid flow in deformable porous
media.

It is well-known that adaptive algorithms for the numerical solution of partial differential equations are nowadays standard tools in science and engineering. A posteriori error estimation, as an essential ingredient of adaptivity, provides adaptive mesh refinement strategy and quantitative estimates of the numerical solution obtained.  For
second-order elliptic problems, the theory of {\it a posteriori } error
estimation has reaches a degree of maturity for finite element of
conforming, nonconforming and mixed types (see
 \cite{Achchab,Ainsworth, Ainsworth 1,Ainsworth;Oden2000,
Babuska,Babuska;Strouboulis2001,Alonso,Bernardi and Verfurth,Braess and
Verfurth, Carstensen0,Carstensen;Bartels,Arnold0,
Carstensen.C;Hu.J;Orlando.A2007,Du;Xie,Kirby,Lovadina and
Stenberg,Verfurth-3} and the references therein). To the authors'
knowledge, no {\it a posteriori } estimation for the MFMFE method has been
proposed in the literature so far.

In this paper, we develop a novel technique to derive
residual-based {\it a posteriori } error estimation for the MFMFE method for
the porous media model in two or three-dimensional case. Since the
MFMFE method employs a special quadrature rule,
 its {\it a posteriori } error
estimator should include a term to control the error of
quadrature.  This is different from the standard analytical technique
based on the discrete $L^{2}$-inner product.  Moreover, we can not
directly utilize the analytical technique developed by {\it
Carstensen} in \cite{Carstensen;Bartels} for nonconforming finite
elements to estimate
\begin{equation*}
\displaystyle\inf_{\beta\in H^{1}(\Omega)}||\nabla\beta-K^{-1}{\bf
u}_{h}||,
\end{equation*}
because  the BDM1  finite
element for  the velocity approximation, ${\bf
u}_{h}$, does not  have the same continuity of mean of
 trace across the interior sides as   the nonconforming finite elements do.
% In fact, the mixed finite
%elements  have only continuity of the
%mean of the normal trace.
To overcome this difficulty, we shall  construct a locally
postprocessed approximation to the pressure solution, obtained by the MFMFE scheme, of a special
auxiliary problem,  and use a derived estimate of quadrature error.
%Note that our
%work here can be understood as a generalization of
%\cite{Carstensen0} to three space dimensions in case of the error of
%the quadrature vanishing, even so, our estimators are also different
%from ones in \cite{Carstensen0} (see remark \ref{newremark}). Moreover, we here consider
%more general boundary conditions with inhomogeneous Dirichlet  boundary and
%homogeneous Neumann one.
We note that the idea of postprocessing in this contribute follows from the works \cite{Lovadina and Stenberg,Vohralik1}.%, Du;XieJCM }. For the convection-diffusion-reaction model
%(\ref{convection-diffufsion-equations1}), Vohral\'{\i}k
%\cite{Vohralik1}, following an idea of postprocessing similar to
%that employed in \cite{Lovadina and Stenberg}, established residual
%{\it a posteriori } error estimates for lowest-order Raviart-Thomas mixed
%finite element discretizations  on simplicial meshes.

The rest of this paper is organized as follows. In section 2, we
introduce some notations and the continuous problem. Section 3 shows the MFMFE method.
 Section 4 includes main results.
Sections 5-6 are respectively devoted to the   a
posteriori error estimation and the analysis of efficiency.
Finally, we illustrate the performance of the obtained estimation in
section 7 by numerical experiments.

\section{Notations and continuous problem}
Let $\mathcal{T}_{h}$ be a shape regular triangulation  of $\Omega\subset{\mathbb{R}}^{d}$ in the sense
of \cite{CIARLET} which satisfies the angle condition, namely there
exists a constant $C_{0}>0$ such that for all $T\in\mathcal{T}_{h}$
\begin{equation*}
C_{0}^{-1}h_{T}^{d}\leq|T|\leq C_{0}h_{T}^{d},
\end{equation*}
where $h_{T} :={\rm diam}(T)$. Let $h$ be a piecewise constant function with $h|_T=h_T$.

We denote by $\varepsilon_{h}$ the set of element sides (or faces) in
$\mathcal{T}_{h}$,  by
$\varepsilon_{T}$ the set of sides (or faces) of element  $T\in\mathcal{T}_{h}$ ,    by
$\varepsilon_{h}^{0}$ and $\varepsilon_{D}$ respectively the sets of the interior
and Dirichlet boundary sides (or faces) of all elements in $\mathcal{T}_{h}$,
 by $\omega_{E}$ the union of all elements in
$\mathcal{T}_{h}$ sharing   side (or face) $E\in\varepsilon_{h}$, and   by $\mathcal{N}$ the set of nodes in
$\mathcal{T}_{h}$.

 For a domain $A\subset\mathbb{R}^{d}$, let
  $(\cdot,\cdot)_{A}$ be the $L^{2}$ inner product on $A$, and
  $<\cdot,\cdot>_{\partial A}$ %be inner product on $\partial A$ for
   the dual pair between $H^{-1/2}(\partial A)$ and
$H^{1/2}(\partial A)$. Let
$W_p^{k}(A)$ be the usual Sobolev space consisting of functions defined on $A$ with all derivatives of order up to  $k$ belonging to $L^p(A)$, with norm
$||\cdot||_{k,p,A}$. When $p=2$, $W_2^{k}(A)=:H^{k}(A)$  and $||\cdot||_{k,2,A}=:||\cdot||_{k,A}$, especially $||\cdot||_{0,A}=:||\cdot||_{A}$ for $k=0$. We omit the subscript $A$     if
$A=\Omega$.  For a tensor-valued function $M=(M_{ij})$, let
$||M||_{\alpha}=\max_{i,j}||M_{ij}||_{\alpha}$ for any norm
$||\cdot||_{\alpha}$. Introduce
\begin{equation*}
 {\bf H}({\rm
div};A):=\{{\bf v}\in L^{2}(A)^d:\nabla\cdot{\bf v}\in L^{2}(A)\},
\end{equation*}
%the space of functions with square-integrable weak divergences.
% and will also
and define the "broken Sobolev space" $$H^{1}(\cup\mathcal{T}_{h})
:=\{\varphi\in L^{2}(\Omega):\varphi|_{T}\in H^{1}(T),\forall
T\in\mathcal{T}_{h}\}. $$
We denote by $[v]|_{E}
:=(v|_{T_{+}})|_{E}-(v|_{T_{-}})|_{E}$ the jump of $v\in
H^{1}(\cup\mathcal{T}_{h})$ over an interior side $E :=T_{+}\cap
T_{-}$ with diameter $h_{E}: ={\rm diam}(E)$,  shared by the two
neighboring (closed) elements $T_{+},T_{-}\in\mathcal {T}_{h}$.
Especially, $[v]|_{E} :=(v|_{T})|_{E}$ if
$E\in\varepsilon_{T}\cap\Gamma_{D}$.

Since we consider two and three-dimensional cases ($d=2,3$) simultaneously, the Curl of a
 function $\psi\in H^{1}(\Omega)^{k}$ with  $k =1$ if $d=2$ and $k =3$ if $d=3$ is 
defined by
\begin{equation*}
{\rm Curl}\psi :=(-\partial_{2}\psi,\partial_{1}\psi)\ \ {\rm if\
d=2}\ \ {\rm and}\ \ {\rm Curl}{\bf \psi} :=\nabla\times{\bf \psi}\ \ {\rm if\
d=3},
\end{equation*}
where $\times$ denotes the usual vector product of two
vectors in $\mathbb{R}^{3}$. Given a unit normal
vector ${\bf n}_{E}=(n_1,\cdots,n_d)^T$ along the side $E$, we define the tangential
component of a vector ${\bf v}\in\mathbb{R}^{d}$ with respect to
${\bf n}_{E}$ by
\begin{equation*}
\gamma_{{\bf t}_{E}}({\bf v}) :=\left \{ \begin{array}{ll}
{\bf v}\cdot(-n_{2},n_{1})  &\quad   \mbox{if}\ \ d=2,\\
{\bf v}\times{\bf n}_{E} &\quad   \mbox{if}\ \ d=3.
 \end{array}\right.
\end{equation*}

Throughout the paper,
$
\nabla_{h} :H^{1}(\cup\mathcal{T}_{h})\rightarrow
(L^{2}(\Omega))^{d}
$ denotes the local version of differential operator
$\nabla$  defined by  $\nabla_{h}\varphi|_{T}
:=\nabla(\varphi|_{T})$  for all $\ T\
\in\mathcal{T}_{h}$.
We also use the notation $A\lesssim B$ to represent
$A\leq CB$ where $C$ is a generic, positive constant    independent  of the mesh size of $\mathcal{T}_{h}$. %depending only on the domain $\Omega$, the shape regularity constant, $C_0$, of $\mathcal{T}_{h}$ and $K$
%,  $C>0$. $C_{i} (i=1,2,\cdots)$ denote  positive constants  , and
 Moreover, $A\approx B$ abbreviates $A\lesssim B\lesssim A$.

Denote
\begin{equation*}
{\bf V}:=\{{\bf v}\in{\bf H}({\rm div};\Omega)\ :{\bf v}\cdot{\bf
n}=0\ \ {\rm on}\ \Gamma_{N}\},\ \ W:=L^{2}(\Omega),
\end{equation*}
then the weak formulation of the model (\ref{equation1}) is as follows:
Find ${\bf u}\in {\bf V}$, $p\in W$ such that
\begin{equation}\label{equation2}
(K^{-1}{\bf u},{\bf v})=(p,\nabla\cdot{\bf v})-<g,{\bf v}\cdot{\bf
n}>_{\Gamma_{D}},\ \ \ \forall\ {\bf v}\in\ {\bf V},
\end{equation}
\begin{equation}\label{equation3}
(\nabla\cdot{\bf u},w)=(f,w),\ \ \ \forall\ {w}\in\ {W}.
\end{equation}
It is well-known that this problem admits a unique solution \cite{Brezzi}.% [14,26].% that (\ref{equation2})-(\ref{equation3})

\section{Multipoint flux mixed finite element method} We follow the notations and definitions employed in
\cite{Yotov06,Yotov10} to describe the MFMFE method.
%Let ${\bf
%V}_{h}\times W_{h}$ be the lowest order ${\rm BDM}_{1}$ mixed finite
%element spaces on the simplicial meshes [13, 14].
Let $\hat{T}$ be the
reference element which is a unit triangle in two-dimensional case or unit
tetrahedron in   three-dimensional case, and   $P_{l}$  be the set of polynomials of degree $\leq l$.  The lowest order
${\rm BDM}_{1}$ mixed finite element spaces on $\hat{T}$ are defined as
\begin{equation*}
\hat{\bf V}(\hat{T})=P_{1}(\hat{T})^{d},\ \ \
\hat{W}(\hat{T})=P_{0}(\hat{T}).
\end{equation*}
Since $\hat{\bf v}\cdot  \hat{\bf n}_{\hat{e}}\in P_1({\hat{e}})$ for any $\hat{\bf v}\in \hat{\bf V}(\hat{T})$ and any edge (or face) $\hat{e}$ of $\hat{T}$,   the degrees of freedom for $\hat{\bf
V}(\hat{T})$ can be chosen to be the values of $\hat{\bf
v}\cdot\hat{\bf n}_{\hat{e}}$ at any two points on each edge
$\hat{e}$ of $\hat{T}$ if $\hat{T}$ is the unit triangle, or any
three points on each face $\hat{e}$ of $\hat{T}$ if $\hat{T}$ is the
unit tetrahedron \cite{Brezzi,Brezzi;Fortin1985}. In the MFMFE method,   these points are chosen to be
the vertices of $\hat{e}$ for the requirement of accuracy and certain
orthogonality   for the trapezoidal
quadrature rules. Such a choice allows for local
velocity elimination and leads to a cell-centered stencil for the
pressure \cite{Yotov06,Yotov10}.

The lowest
order ${\rm BDM}_{1}$ spaces on $\mathcal{T}_{h}$ are given by
\begin{equation*}
\begin{array}{lll}
{\bf V}_{h} &:&=\{{\bf v}\in {\bf V}\ :\ \ \ {\bf
v}|_{T}=\frac{1}{J_{T}}DF_{T}\hat{\bf v}\circ F_{T}^{-1},\ \ \hat{\bf v}\in\hat{\bf
V}(\hat{T})\ \ \ \forall\ T\in\mathcal{T}_{h}\},\vspace{2mm}\\
{W}_{h} &:&=\{w\in {W}\ :\ \ \ {w}|_{T}=\hat{w}\circ F_{T}^{-1},\ \
\hat{w}\in\hat{W}(\hat{T})\ \ \ \forall\ T\in\mathcal{T}_{h}\},
\end{array}
\end{equation*}
where $F_{T}^{-1}$ is the inverse mapping of the bijection $F_{T} :
\hat{T}\rightarrow T$, $DF_{T}$ is the Jacobian matrix with respect
to $F_{T}$ on the element $T$ with  $J_{T}=|det(DF_{T})|$.  Note that the vector transformation ${\bf
v}=\frac{1}{J_{T}}DF_{T}\hat{\bf v}\circ F_{T}^{-1}$ is is known as the Piola
transformation.

For ${\bf q},{\bf v}\in{\bf V}_{h}$, it holds
\begin{equation*}
\begin{array}{lll}
\displaystyle\int_{T}K^{-1}{\bf q}\cdot{\bf v}d{\bf
x}&=&\displaystyle\int_{\hat{T}}\hat{K}^{-1}\frac{1}{J_{T}}DF_{T}\hat{\bf
q}\cdot\frac{1}{J_{T}}DF_{T}\hat{\bf v}J_{T}d\hat{\bf
x}\vspace{2mm}\\
&=&\displaystyle\int_{\hat{T}}\frac{1}{J_{T}}(DF_{T})^{\rm
T}\hat{K}^{-1}DF_{T}\hat{\bf q}\cdot\hat{\bf v}d\hat{\bf
x}\vspace{2mm}\\
&=&\displaystyle\int_{\hat{T}}\mathcal{K}^{-1}\hat{\bf
q}\cdot\hat{\bf v}d\hat{\bf x}
\end{array}
\end{equation*}
with
$
\mathcal{K} :=J_{T}DF_{T}^{-1}\hat{K}(DF_{T}^{-1})^{\rm T}.
$
The quadrature formula on an element $T$ is then defined as \cite{Yotov06,Yotov10}
\begin{equation}\label{quadrature}
(K^{-1}{\bf q},{\bf v})_{Q,T} :=(\mathcal{K}^{-1}\hat{\bf
q},\hat{\bf v})_{\hat{Q},\hat{T}}
:=\frac{|\hat{T}|}{s}\sum\limits_{i=1}^{s}\mathcal{K}^{-1}(\hat{\bf
r}_{i})\hat{\bf q}(\hat{\bf r}_{i})\cdot\hat{\bf v}(\hat{\bf
r}_{i}),
\end{equation}
where $\hat{\bf r}_{i}$ ($i=1,2,\cdots,s$) are the corresponding vertices of $\hat{T}$ with
$s=3$ for the unit triangle  and $s=4$ for the unit tetrahedron.

Define the global quadrature
formula as
\begin{equation}\label{quadrature-global}
(K^{-1}{\bf q},{\bf
v})_{Q}=\sum\limits_{T\in\mathcal{T}_{h}}(K^{-1}{\bf q},{\bf
v})_{Q,T},
\end{equation}
%The integration on any element $T$ is performed by mapping to the
%reference element $\hat{T}$, and the quadrature rule is defined on
%$\hat{T}$. Using the definition of the finite element spaces, we
%have
then the MFMFE method is formulated as follows: Find ${\bf u}_{h}\in
{\bf V}_{h}$ and $p_{h}\in W_{h}$ such that
\begin{equation}\label{equation4}
(K^{-1}{\bf u}_{h},{\bf v}_{h})_{Q}=(p_{h},\nabla\cdot{\bf
v}_{h})-<g,{\bf v}_{h}\cdot{\bf n}>_{\Gamma_{D}},\ \ \ \forall\ {\bf
v}_{h}\in\ {\bf V}_{h},
\end{equation}
\begin{equation}\label{equation5}
(\nabla\cdot{\bf u}_{h},w_{h})=(f,w_{h}),\ \ \ \forall\ w_{h}\in\
W_{h}.
\end{equation}
The existence and uniqueness of the solution to the
scheme (\ref{equation4})-(\ref{equation5}) follow from   \cite{Yotov06,Yotov10}.
%Although we do not describe
%how the MFMFE method reduces a system for the pressures at the cell
%centers, because this is not our emphasis in this paper.
As shown in  \cite{Yotov06,Yotov10},
  the algebraic system that arises from (\ref{equation4})-(\ref{equation5}) is of the form
\begin{equation}\label{equation6}
\left(\begin{array}{c}\ \ A\ \ \ \ B^{\rm T}\\-B\ \ \ 0
\end{array}\right)\left(\begin{array}{c}\ U\ \\ \ P
\end{array}\right)=\left(\begin{array}{c}\ G\\ \ F
\end{array}\right),
\end{equation}
where $A=(a_{ij}),\ B=(b_{lj})$ with $a_{ij}=(K^{-1}{\bf v}_{j},{\bf v}_{i})_{Q}$ and
$b_{lj}=-(\nabla\cdot{\bf v}_{j}, w_{l})$, and $\{{\bf v}_{i}\}$,
$\{w_{l}\}$ are respectively the bases of ${\bf V}_{h}$ and $W_{h}$.
The matrix $A$ is block-diagonal with symmetric and positive definite blocks, and
the local elimination of $U$   leads to a system for $P$ with a symmetric
and positive definite matrix $BA^{-1}B^{T}$. For the details, we
refer to \cite{Yotov06,Yotov10}.

\section{Main results}  Let
 $\eta_{h}$ be the discretization indicator defined  by
\begin{equation}\label{eta-h}
\eta_{h}^{2} :=||h(f-\nabla\cdot{\bf
u}_{h})||^{2}+\sum\limits_{T\in\mathcal{T}_{h}}\sum\limits_{E\in\varepsilon_{T}}h_{E}J_{{\bf
t}_E}^{2},
\end{equation}
where
\begin{equation}\label{J-tE}
J_{{\bf t}_E}^{2} :=\left \{ \begin{array}{ll}
||[\gamma_{{\bf t}_{E}}(K^{-1}{\bf u}_{h})]||_{E}^{2} &   \mbox{if}\ \ E\in\varepsilon_{h}^{0}\cap\partial T,\\
||\gamma_{{\bf t}_{E}}(K^{-1}{\bf u}_{h})-\partial g/\partial
s||_{E}^{2}+h_{E}^{2}||\frac{\partial^{2}g}{\partial
s^{2}}||_{E}^{2}&   \mbox{if}\ \ E\in\partial
T\cap\varepsilon_{D},\\
0&  \mbox{if}\ \ E\in\partial T\cap\Gamma_{N},
 \end{array}\right.
\end{equation}
and $\partial g/\partial s$ and
$\partial^{2}g/\partial s^{2}$ denote respectively      the first and
second order tangential derivatives of function $g\in H^{2}(E)$ along   side $E$.
Introduce the quadrature indicator
\begin{equation}\label{eta-Q}
\eta_{Q}^{2} :=\sum\limits_{T\in\mathcal{T}_{h}}h_{T}^{2}||{\bf
u}_{h}||_{1,T}^{2}.
\end{equation}
We note this indicator is  owing to the use of
the special quadrature formula (\ref{quadrature}) in the MFMFE method.

We now state in Theorems \ref{equation7}-\ref{newpressure}  {\it a posteriori } error estimates for the  errors of
velocity   and pressure in $L^{2}-$norm, respectively.

\begin{theorem}\label{equation7}
Let $({\bf u},p)\in{\bf V}\times W$ be the weak solution of the
continuous problem (\ref{equation2})-(\ref{equation3}), and $({\bf
u}_{h},p_{h})\in{\bf V}_{h}\times W_{h}$ be the solution of the
MFMFE method (\ref{equation4})-(\ref{equation5}). Assume $K^{-1}\in W_{\infty}^{1}(\mathcal{T}_{h})$. Then it holds
%, $\eta_{h}$ and $\eta_{Q}$ be the discretization and quadrature indicators, respectively. % which only
%depends on $K,\Omega$, and the shape regularity of the elements
\begin{equation}\label{equation8}
||K^{-1/2}({\bf u}-{\bf u}_{h})||\lesssim (\eta_{h}^{2}+\eta_{Q}^{2})^{1/2}.
\end{equation}
\end{theorem}

\begin{theorem}\label{newpressure} Assume $K^{-1}\in W_{\infty}^{2}(\mathcal{T}_{h})$. 
Under the
assumptions of Theorem \ref{equation7}, it holds
\begin{equation}\label{newpressure1}
||Q_{h}p-p_{h}||\lesssim h_{\rm max}(\eta_{h}+\eta_{Q})+||h(f-\nabla\cdot{\bf u}_{h})||,
\end{equation}
\begin{equation}\label{newpressure2}
||p-p_{h}||\lesssim h_{\rm max}(\eta_{h}+\eta_{Q})+||hK^{-1}{\bf
u}_{h}||+||h(f-\nabla\cdot{\bf u}_{h})||.
\end{equation}
Here  $h_{\rm max} :=\max_{T\in\mathcal{T}_{h}}h_{T}$, and $Q_{h}$ denotes the $L^{2}-$projection operator
onto $W_{h}$. 
\end{theorem}

\begin{remark}\ We note that the two terms $||h(f-\nabla\cdot{\bf u}_{h})||$ and $\displaystyle\{\sum\limits_{E\in\varepsilon_{D}}h_{E}^{3}||\frac{\partial^{2}g}{\partial s^{2}}||_{E}^{2}\}^{1/2}$
in $\eta_{h}$ in the estimator $\eta_{h}$ are of high order with respect to the
lowest order scheme, which  are usually omitted in
computation. In fact, from (\ref{equation5}) it follows $\nabla\cdot{\bf u}_{h}=Q_{h}f$, and
$||h(f-\nabla\cdot{\bf u}_{h})||=||h(f-Q_hf)||$ turns out to be an oscillation term of high order.
%for the
%lowest order scheme. In addition, the term
%$h_{E}^{3}||\frac{\partial^{2}g}{\partial s^{2}}||_{E}^{2}$ for all
%$E\in\varepsilon_{D}$ is still of high order with respect to the
%lowest order scheme. Therefore, these terms
\end{remark}

\begin{remark}\label{newremark} The above  estimates (\ref{equation8})-(\ref{newpressure2}) also apply to the original mixed finite element discretization where the special quadrature rule (\ref{quadrature}) is not used in the
 scheme (\ref{equation4})-(\ref{equation5}). In this case,  the estimator $\eta_{Q}$ is not involved, and then $\eta_{Q}=0$ in the estimates (\ref{equation8})-(\ref{newpressure2}). In this sense, our work
can be regarded as a generalization of Carstensen's \cite{Carstensen0} to
the three-dimensional case.  We note that   our estimator $\eta_h$ is a bit different
from that in \cite{Carstensen0} due to no occurrence of the term
$||h{\rm Curl}_{h}(K^{-1}{\bf u}_{h})||$ (${\rm Curl}_{h}$ denotes the piecewise ${\rm Curl}$ operator acting on element by element in $\mathcal{T}_{h}$). Here we also
consider more general boundary conditions.
\end{remark}

We finally state in Theorem \ref{equation9} the efficiency of the {\it a posteriori } error
estimators. Note that the
efficiency of a reliable   {\it a posteriori } error estimator  means that its converse estimate holds up to high
order terms and different multiplicative constants. For the sake of
simplicity, we assume that $K^{-1}$ is a matrix of piecewise polynomial
functions.

\begin{theorem}\label{equation9}
Under the assumptions of Theorems \ref{equation7}-\ref{newpressure}, it holds
\begin{equation*}%\label{equation10}
\eta_{h}+\eta_{Q}+h_{\rm max}^{-1}||hK^{-1}{\bf u}_{h}||\lesssim ||K^{-1/2}({\bf u}-{\bf
u}_{h})||+||h^{-1}(p-p_{h})||+h.o.t..
\end{equation*}
where $ h.o.t.$ denotes some high-order term depending on given data.
%\begin{equation*}%\label{pressureefficincy}
%h_{\rm max}(\eta_{h}+\eta_{Q})+||hK^{-1}{\bf u}_{h}||\lesssim h_{\rm max}\left(||K^{-1/2}({\bf u}-{\bf
%u}_{h})||+||h^{-1}(p-p_{h})||\right)+h.o.t..
%\end{equation*}
\end{theorem}

\section{A posteriori error analysis} This section is devoted to the proofs of Theorems \ref{equation7}-\ref{newpressure}.

Introduce the global quadrature error $\sigma(K^{-1}{\bf
u}_{h},{\bf v}_{h})$ and the element quadrature error $\sigma_{T}(K^{-1}{\bf u}_{h},{\bf v}_{h})$ as follows:
\begin{equation}\label{quadrature-error}
\sigma(K^{-1}{\bf u}_{h},{\bf v}_{h})|_{T}= \sigma_{T}(K^{-1}{\bf u}_{h},{\bf v}_{h}) :=(K^{-1}{\bf
u}_{h},{\bf v}_{h})_{T}-(K^{-1}{\bf u}_{h},{\bf v}_{h})_{Q,T}, ,\ \ {\rm for\ all}\ \ T\in\mathcal{T}_{h}.
\end{equation}
%\begin{equation*}
%\sigma(K^{-1}{\bf u}_{h},{\bf v}_{h})|_{T}=\sigma_{T}(K^{-1}{\bf
%u}_{h},\bf{v}_{h}),\ \ {\rm for\ all}\ \ T\in\mathcal{T}_{h}.
%\end{equation*}
Let ${\bf V}_{h}^{0} :={\rm RT}_{0}(\mathcal{T}_{h})$ denote the
lowest order ${\rm RT}$ element space on $\mathcal{T}_{h}$.

We state two estimates on the quadrature error derived in \cite{Yotov06,Yotov10} as follows. If $K^{-1}\in
W_{\infty}^{1}(T)$ for all element $T\in\mathcal{T}_{h}$, then it
holds
\begin{equation}\label{equation12+2}
|\sigma(K^{-1}{\bf q}_{h},{\bf v}_{h})|\lesssim
\sum\limits_{T\in\mathcal{T}_{h}}h_{T}||{\bf q}_{h}||_{1,T}||{\bf
v}_{h}||_{T}
\end{equation}
for all ${\bf q}_{h}\in{\bf V}_{h}$, ${\bf v}_{h}\in{\bf
V}_{h}^{0}$. Moreover, if $K^{-1}\in W_{\infty}^{2}(T)$ for all
element $T\in\mathcal{T}_{h}$, then it holds
\begin{equation}\label{pressure7}
|\sigma(K^{-1}{\bf q}_{h},{\bf v}_{h})|\lesssim
\sum\limits_{T\in\mathcal{T}_{h}}h_{T}^{2}||{\bf
q}_{h}||_{1,T}||{\bf v}_{h}||_{1,T}
\end{equation}
for all ${\bf q}_{h},{\bf v}_{h}\in{\bf V}_{h}$.
%Note that the estimates
%(\ref{equation12+2}) and (\ref{pressure7}) can be found in
%.

Denote respectively by $\Pi$ and $\Pi_{0}$ the standard projection operators  from ${\bf H}({\rm div};\Omega)\cap
(L^{\varrho}(\Omega))^{d}$ onto $V_{h}$ and
$V_{h}^{0}$  for some $\varrho>2$  (cf.
\cite{Carstensen0,Yotov06}). It holds the following estimates:
\begin{equation}\label{pressureadd}
||h^{-1}({\bf q}-\Pi_{0}{\bf q})||\lesssim||{\bf
q}||_{H^{1}(\cup\mathcal{T}_{h})}\ \ \ {\rm for\ all}\ \ {\bf q}\in
(H^{1}(\cup\mathcal{T}_{h}))^{d}\cap{\bf H}({\rm div};\Omega),
\end{equation}
\begin{equation}\label{equation12+3}
||\Pi_{0}{\bf v}||_{1,T}\lesssim||{\bf v}||_{1,T},||\Pi{\bf
v}||_{1,T}\lesssim||{\bf v}||_{1,T}\ \ \ {\rm for \ all}\ \ {\bf
v}\in (H^{1}(T))^{d},\ \ \forall T\in\mathcal{T}_{h}.
\end{equation}
Note that bound (\ref{pressureadd}) can be found in
\cite{Carstensen0}, and bounds (\ref{equation12+3}) are the direct
results of Lemma 3.1 in \cite{Yotov06}.

To derive a  reliable {\it a posteriori } error estimate for the velocity error, we need to
%we need a series of auxiliary results derived by constructing
%several auxiliary problems and postprocessing approximation.
introduce  an auxiliary problem as following:
\begin{equation}\label{equation13}
 \left \{ \begin{array}{ll}
    \nabla\cdot(K\nabla\vartheta)=\nabla\cdot{\bf u}_{h} &\quad   \mbox{in}\ \ \Omega,\\
\vartheta=-g  &\quad \mbox{on}\ \ \Gamma_{D},\\
K\nabla\vartheta\cdot{\bf n}=0  &\quad \mbox{on}\ \ \Gamma_{N}.
\end{array}\right.
\end{equation}
Since $K$ is a symmetric and uniformly positive definite tensor, by the
Lax-Milgram theorem there exists a unique solution $\vartheta\in
H^{1}(\Omega)$ to this problem, provided that $g\in H^{1/2}(\Gamma_D)$. As
$K\nabla\vartheta-{\bf u}_{h}$ is divergence-free,   a
decomposition of two or three-dimensional vector fields  (see Theorem 3.4 and
Remark 3.10 in \cite{Girault})  implies that
there exists a stream function $\psi\in H^{1}(\Omega)^{k}$ such that
$$
K\nabla\vartheta-{\bf u}_{h}={\rm Curl}\ \psi.
$$
Since $K\nabla\vartheta\cdot{\bf n}$ and ${\bf u}_{h}\cdot{\bf n}$ vanish on
$\Gamma_{N}$, we easily know ${\rm Curl}\ \psi\cdot{\bf n}=0$ on
$\Gamma_{N}$.

Introduce $H_{D}^{1}(\Omega): =\{v\in H^{1}(\Omega): v=0\ \ {\rm on}\ \
\Gamma_{D}\}$, then $z: =-(p+\vartheta)\in H_{D}^{1}(\Omega)$
and it holds
\begin{equation}\label{decom}
{\bf u}-{\bf u}_{h}=-K\nabla p-K\nabla\vartheta+{\rm Curl}\
\psi=K\nabla z+{\rm Curl}\ \psi.
\end{equation}
This relation leads to
\begin{equation}\label{equation14}
\begin{array}{lll}
||K^{-1/2}({\bf u}-{\bf
u}_{h})||^{2}&=&\displaystyle\int_{\Omega}K^{-1}({\bf u}-{\bf
u}_{h})\cdot({\bf u}-{\bf u}_{h})\vspace{2mm}\\
&=&\displaystyle\int_{\Omega}(\nabla z+K^{-1}{\rm Curl}\
\psi)\cdot(K\nabla z+{\rm Curl}\ \psi)\vspace{2mm}\\
&=&\displaystyle\int_{\Omega}K\nabla z\cdot\nabla
z+2\int_{\Omega}\nabla z\cdot{\rm Curl}\
\psi+\int_{\Omega}K^{-1}{\rm Curl}\ \psi\cdot{\rm Curl}\ \psi.
\end{array}
\end{equation}
Using integration by parts and noticing ${\rm
Curl}\ \psi\cdot{\bf n}=0$ on $\Gamma_{N}$ and $z=0$ on
$\Gamma_{D}$, we have
\begin{equation}\label{equation15}
\displaystyle\int_{\Omega}\nabla z\cdot{\rm Curl}\
\psi=-\displaystyle\int_{\Omega}\nabla\cdot({\rm Curl}\
\psi)z+\int_{\Gamma_{D}\cup\Gamma_{N}}{\rm Curl}\ \psi\cdot{\bf
n}z=0.
\end{equation}
Notice that $K\nabla z=({\bf u}-{\bf u}_{h})-{\rm Curl}\ \psi$,
$({\bf u}-{\bf u}_{h})\cdot{\bf n}=0$ on $\Gamma_{N}$ and $z=0$
  on $\Gamma_{D}$. The relation (\ref{equation15}) and integration by
parts yield
\begin{equation}\label{equation16}
%\begin{array}{lll}
\displaystyle\int_{\Omega}K\nabla z\cdot\nabla
z=\displaystyle\int_{\Omega}({\bf u}-{\bf u}_{h})\cdot\nabla
z
=\displaystyle-\int_{\Omega}\nabla\cdot({\bf u}-{\bf u}_{h})z.
%\end{array}
\end{equation}
Let $Q_{h}z$ denote the $L^{2}-$projection of $z$ onto $W_{h}$. From
(\ref{equation3}) and (\ref{equation5}) it follows
\begin{equation}\label{equation17}
(\nabla\cdot({\bf u}-{\bf u}_{h}),Q_{h}z)=0.
\end{equation}
In view of $\nabla\cdot{\bf u}=f$, the above two relations, (\ref{equation16}) and (\ref{equation17}), imply
\begin{equation*}
\begin{array}{lll}
\displaystyle\int_{\Omega}K\nabla z\cdot\nabla
z&=&\displaystyle-\int_{\Omega}\nabla\cdot({\bf u}-{\bf
u}_{h})(z-Q_{h}z)\vspace{2mm}\\
&=&\displaystyle\sum\limits_{T\in\mathcal{T}_{h}}\int_{T}(-f+\nabla\cdot{\bf
u}_{h})(z-Q_{h}z)\vspace{2mm}\\
&\lesssim&\displaystyle\sum\limits_{T\in\mathcal{T}_{h}}h_{T}||f-\nabla\cdot{\bf
u}_{h}||_{T}||\nabla z||_{T}\vspace{2mm}\\
&\lesssim&||h(f-\nabla\cdot{\bf u}_{h})||\ ||K^{1/2}\nabla z||,
\end{array}
\end{equation*}
which results in 
\begin{equation}\label{equation18}
||K^{1/2}\nabla z||
\lesssim ||h(f-\nabla\cdot{\bf u}_{h})||.
\end{equation}

By (\ref{decom}) and 
(\ref{equation15}) we have
\begin{equation}\label{relation-id}
||K^{-1/2}({\bf u}-{\bf u}_{h})||^{2}=||K^{1/2}\nabla
z||^{2}+||K^{-1/2}{\rm Curl}\ \psi||^{2}.
\end{equation}
%which leads to
%\begin{equation}\label{equation19}
%||K^{1/2}\nabla z||\leq||K^{-1/2}({\bf u}-{\bf u}_{h})||\ \ \ {\rm
%and} \ \ \ ||K^{-1/2}{\rm Curl}\ \psi||\leq||K^{-1/2}({\bf u}-{\bf
%u}_{h})||.
%\end{equation}

Recalling $\displaystyle\int_{\Omega}{\rm Curl}\ \psi\cdot\nabla
v=0$ for all $v\in H_{D}^{1}(\Omega)$, in light of (\ref{decom}) we have, for any $\beta\in
H^{1}(\Omega)$,
\begin{equation*}
\begin{array}{lll}
& &\displaystyle\int_{\Omega}K^{-1}{\rm Curl}\ \psi\cdot{\rm Curl}\
\psi=\displaystyle\int_{\Omega}(K^{-1}({\bf u}-{\bf u}_{h})-\nabla
z)\cdot{\rm Curl}\ \psi\vspace{2mm}\\
%& &\ =\displaystyle\int_{\Omega}K^{-1}({\bf u}-{\bf u}_{h})\cdot{\rm Curl}\ \psi\vspace{2mm}\\
& &\ =\displaystyle\int_{\Omega}K^{-1}({\bf u}-{\bf u}_{h}-K\nabla v)\cdot{\rm Curl}\ \psi\vspace{2mm}\\
& &\ =\displaystyle\int_{\Omega}K^{-1}({\bf u}-K\nabla v-K\nabla
\beta)\cdot{\rm Curl}\ \psi+\int_{\Omega}K^{-1}(K\nabla\beta-{\bf
u}_{h})\cdot{\rm Curl}\ \psi\vspace{2mm}\\
& &\ \leq\displaystyle(||K^{-1}({\bf u}-K\nabla v-K\nabla
\beta)||+||\nabla\beta-K^{-1}{\bf u}_{h}||)||{\rm Curl}\ \psi||,
\end{array}
\end{equation*}
which implies
\begin{equation}\label{equation20}
||K^{-1/2}{\rm Curl}\ \psi||
\lesssim  \inf_{v\in H_{D}^{1}(\Omega)}||K^{-1}({\bf
u}-K\nabla v-K\nabla\beta)||+\inf_{\beta\in
H^{1}(\Omega)}||\nabla\beta-K^{-1}{\bf u}_{h}||.
\end{equation}

Finally, from 
(\ref{equation18})-(\ref{equation20}) it follows
%\begin{equation}\label{equation21}
%\begin{array}{lll}
%||K^{-1/2}({\bf u}-{\bf u}_{h})||^{2}&\lesssim&\displaystyle
%\{||h(f-\nabla\cdot{\bf u}_{h})||+\inf_{v\in
%H_{D}^{1}(\Omega)}||K^{-1}({\bf u}-K\nabla v-K\nabla
%\beta)||\vspace{2mm}\\
%& &\ \displaystyle+\inf_{\beta\in
%H^{1}(\Omega)}||\nabla\beta-K^{-1}{\bf u}_{h}||\}||K^{-1/2}({\bf
%u}-{\bf u}_{h})||.
%\end{array}
%\end{equation}
%Dividing by $||K^{-1/2}({\bf u}-{\bf u}_{h})||$, we obtain from
%(\ref{equation21})
\begin{equation}\label{equation22}
\begin{array}{lll}
||K^{-1/2}({\bf u}-{\bf u}_{h})||&\lesssim&\left\{\inf_{v\in H_{D}^{1}(\Omega)}||K^{-1}({\bf u}-K\nabla
v-K\nabla\beta)||\right.\vspace{2mm}\\
& &\left. \displaystyle+\inf_{\beta\in
H^{1}(\Omega)}||\nabla\beta-K^{-1}{\bf
u}_{h}||+||h(f-\nabla\cdot{\bf u}_{h})||\right\}.
\end{array}
\end{equation}

In what follows, we shall follow the routines of \cite{Carstensen;Bartels} to estimate the first and second terms on the
right-hand side of (\ref{equation22}). To this end, we assume that
$g\in H^{1}(\Gamma_{D})\cap C(\Gamma_{D})$ and $g|_{E}\in H^{2}(E)$
for all $E\in\varepsilon_{h}\cap\Gamma_{D}$ and denote by $g_{h,D}$
the nodal $\varepsilon_{D}-$piecewise linear interpolation of $g$
on $\Gamma_{D}$ which satisfies $g_{h,D}({\bf z})=g({\bf z})$ for all
${\bf z}\in\mathcal{N}\cap\Gamma_{D}$. Let $\{\varphi_{\bf z}: {\bf z}\in\mathcal{N}\}$ be
 the nodal basis of the lowest order finite element space associated to $\mathcal{T}_{h}$,
 i.e., $\varphi_{\bf z}\in C(\bar{\Omega}), \varphi_{\bf z}|_{T}\in P_{1}(T)$ for all $T\in\mathcal{T}_{h}$,
 $\varphi_{\bf z}({\bf x})=0$ for ${\bf x}\in\mathcal{N}/\{{\bf z}\}$, and $\varphi_{\bf z}({\bf z})=1$.
 Denote by  $\omega_{\bf z} :={\rm int}({\rm supp}\varphi_{\bf z})$. We then introduce a subspace of
$H^{1}(\Omega)$, $\tilde{S}$, as follows (see
\cite{Carstensen;Bartels}):
$$
\tilde{S} :==\left\{ \begin{array}{c} \displaystyle\sum\limits_{{\bf z}\in\mathcal{N}}\varphi_{\bf z}v_{\bf z}: \forall\ {\bf z}
\in\mathcal{N},v_{\bf z}\in C(\omega_{\bf z}),\ v_{\bf z}|_{\omega_{\bf z}}\ {\rm is\ a\ piecewise} \\ {\rm polynomial,\ and}\ v_{\bf z}=-g_{h,D}\ {\rm on}\ \Gamma_{D}\cap\omega_{\bf z}.
\end{array}\right\}
$$

\begin{lemma}\label{equation23}
For $\beta\in\tilde{S}$, it holds
\begin{equation}\label{equation24}
\inf_{v\in H_{D}^{1}(\Omega)}||K^{-1}({\bf u}-K\nabla v-K\nabla
\beta)||\lesssim\{\sum\limits_{E\subset\Gamma_{D}}h_{E}^{3}||\partial^{2}g/\partial
s^{2}||_{E}^{2}\}^{1/2}.
\end{equation}
\end{lemma}
\begin{proof} The definition of $\tilde{S}$ shows $\beta=-g_{h,D}$
on $\Gamma_{D}$. Noticing
$K^{-1}{\bf u}=-\nabla p$, we have
\begin{equation*}
\inf_{v\in H_{D}^{1}(\Omega)}||K^{-1}({\bf u}-K\nabla v-K\nabla
\beta)||=\inf_{w\in H^{1}(\Omega),w|_{\Gamma_{D}}=g-g_{h,D}}||\nabla
w||.
\end{equation*}
The desired result (\ref{equation24}) immediately follows from an
estimate in the proof of Lemma 3.4 in \cite{Carstensen;Bartels}.
\end{proof}

On the other hand, it holds 
\begin{equation}\label{equation25}
\inf_{\beta\in H^{1}(\Omega)}||\nabla\beta-K^{-1}{\bf
u}_{h}||\leq\inf_{v_{h}\in\tilde{S}}||\nabla v_{h}-K^{-1}{\bf
u}_{h}||.
\end{equation}
It is sophisticated to give a computational upper bound for the
right-hand side term of (\ref{equation25}) with the help of ${\bf
u}_{h}$ and given data. To this end, let $\overline{K^{-1}}$ denote
the piecewise mean value of $K^{-1}$ on $\mathcal{T}_{h}$, i.e. $\overline{K^{-1}}|_T=\frac{1}{|T|}\int_{T} K^{-1}({\bf x})  d{\bf x}$ for all $T\in\mathcal{T}_{h}.$  Then 
%verify
%\begin{equation*}
%\xi^{\rm
%T}\overline{K^{-1}}|_{T}\xi=\displaystyle\frac{1}{|T|}\int_{T}\xi^{\rm
%T}K^{-1}({\bf x})\xi d{\bf x},\ \ \ {\rm for\ all\ vectors}\ \
%\xi\in\mathbb{R}^{d}, T\in\mathcal{T}_{h}.
%\end{equation*}
$\overline{K^{-1}}$ is symmetric and has the following $V-$ellipticity:
\begin{equation*}
k_{1}^{-1}\xi^{\rm T}\xi\leq\xi^{\rm T}\overline{K^{-1}}\xi\leq
k_{0}^{-1}\xi^{\rm T}\xi\ \ \ {\rm for\ all}\ {\bf x}\in\Omega,\ \
\xi\in\mathbb{R}^{d}.
\end{equation*}

Recall that ${\bf V}_{h}^{0}$ is the lowest order ${\rm RT}$ element
space on $\mathcal{T}_{h}$.
 and $W_{h}$ is the piecewise constant space.% $W_{h}^{0} :=W_{h}$. 
Introduce the following auxiliary problem: Find $(\tilde{\bf
u}_{h},\tilde{p}_{h})\in{\bf V}_{h}^{0}\times W_{h}$ such that
\begin{equation}\label{equation26}
(\overline{K^{-1}}\tilde{\bf u}_{h},{\bf
v}_{h})=(\tilde{p}_{h},\nabla\cdot{\bf v}_{h})-<g,{\bf
v}_{h}\cdot{\bf n}>_{\Gamma_{D}},\ \ \forall\ {\bf v}_{h}\in{\bf
V}_{h}^{0},
\end{equation}
\begin{equation}\label{equation27}
(\nabla\cdot\tilde{\bf u}_{h},w_{h})=(f,w_{h}),\ \ \ \forall\
w_{h}\in W_{h}.
\end{equation}
It is well-known that this problem admits a unique solution (see \cite{Brezzi}).
\begin{lemma}\label{equation27+1}
Let $(\tilde{\bf u}_{h},\tilde{p}_{h})\in{\bf V}_{h}^{0}\times
W_{h}$ be the solution of the auxiliary problem (\ref{equation26})-(\ref{equation27}), and $({\bf u}_{h},p_{h})\in{\bf V}_{h}\times
W_{h}$ be the solution of the MFMFEM scheme (\ref{equation4})-(\ref{equation5}).
Assume $K^{-1}\in W_{\infty}^{1}(\mathcal{T}_{h})$. Then it holds
\begin{equation}\label{equation27+2}
||\overline{K^{-1}}^{1/2}(\tilde{\bf u}_{h}-\Pi_{0}{\bf
u}_{h})||\lesssim\{\sum\limits_{T\in\mathcal{T}_{h}}h_{T}^{2}||{\bf
u}_{h}||_{1,T}^{2}\}^{1/2},
\end{equation}
where $\Pi_{0}$ is the standard projection operator from ${\bf H}({\rm div};\Omega)$ onto ${\bf V}_{h}^{0}$.
\end{lemma}
\begin{proof}
Notice that ${\bf V}_{h}^{0}\subset{\bf V}_{h}$. From
(\ref{equation4}) we get
\begin{equation}\label{equation28}
\begin{array}{lll}
(\overline{K^{-1}}\Pi_{0}{\bf u}_{h},{\bf v}_{h})
&=&(p_{h},\nabla\cdot{\bf v}_{h})-<g,{\bf v}_{h}\cdot{\bf
n}>_{\Gamma_{D}}\vspace{2mm}\\
& &\ +(\overline{K^{-1}}\Pi_{0}{\bf u}_{h},{\bf v}_{h})-(K^{-1}{\bf
u}_{h},{\bf v}_{h})_{Q},\ \forall\ {\bf v}_{h}\in{\bf V}_{h}^{0}.
\end{array}
\end{equation}
Using the commuting property of $\Pi_{0}$ and
(\ref{equation5}), we have 
\begin{equation}\label{equation29}
%\begin{array}{lll}
(\nabla\cdot\Pi_{0}{\bf u}_{h},w_{h})=(Q_{h}\nabla\cdot{\bf
u}_{h},w_{h})%\vspace{2mm}\\
=(\nabla\cdot{\bf u}_{h},w_{h})=(f,w_{h}),\ \ \ \forall
w_{h}\in W_{h}.
%\end{array}
\end{equation}
A combination of (\ref{equation27})
and (\ref{equation29}) yields
\begin{equation}\label{equation29-1}
(\nabla\cdot(\tilde{\bf u}_{h}-\Pi_{0}{\bf u}_{h}),w_{h})=0,\ \ \ \forall
w_{h}\in W_{h}.
\end{equation}
Taking  ${\bf v}_{h}=\tilde{\bf u}_{h}-\Pi_{0}{\bf u}_{h} \in{\bf V}_{h}^{0}$,  subtracting (\ref{equation28}) from (\ref{equation26}) and using (\ref{equation29-1}), we have
\begin{equation}\label{equation30}
\begin{array}{lll}
& &\displaystyle||\overline{K^{-1}}^{1/2}(\tilde{\bf
u}_{h}-\Pi_{0}{\bf u}_{h})||^{2}=(\overline{K^{-1}}(\tilde{\bf
u}_{h}-\Pi_{0}{\bf
u}_{h}),\tilde{\bf u}_{h}-\Pi_{0}{\bf u}_{h})\vspace{2mm}\\
& &\ =\displaystyle(\tilde{p}_{h}-p_{h},\nabla\cdot(\tilde{\bf
u}_{h}-\Pi_{0}{\bf u}_{h}))+(K^{-1}{\bf u}_{h},{\bf
v}_{h})_{Q}-(\overline{K^{-1}}\Pi_{0}{\bf u}_{h},{\bf v}_{h})\vspace{2mm}\\
& &\ =\displaystyle(K^{-1}{\bf u}_{h},{\bf v}_{h})_{Q}-(K^{-1}{\bf
u}_{h},{\bf v}_{h})+(K^{-1}{\bf u}_{h},{\bf v}_{h})
-(\overline{K^{-1}}\Pi_{0}{\bf u}_{h},{\bf v}_{h})\vspace{2mm}\\
& &\ =\displaystyle-\sigma(K^{-1}{\bf u}_{h},{\bf v}_{h})+((K^{-1}-\overline{K^{-1}}){\bf
u}_{h},{\bf v}_{h})+(\overline{K^{-1}}({\bf u}_{h}-\Pi_{0}{\bf
u}_{h}),{\bf v}_{h}).
\end{array}
\end{equation}
The work left is to estimate the three terms in the last line of (\ref{equation30}).
Notice that the inequality (\ref{equation12+2}) implies
\begin{equation}\label{equation31}
\begin{array}{lll}
& &|-\sigma(K^{-1}{\bf u}_{h},\tilde{\bf u}_{h}-\Pi_{0}{\bf u}_{h})|
\lesssim\displaystyle\sum\limits_{T\in\mathcal{T}_{h}}h_{T}||{\bf
u}_{h}||_{1,T}||\tilde{\bf u}_{h}-\Pi_{0}{\bf
u}_{h}||_{T}\vspace{2mm}\\
& &\displaystyle\lesssim
\{\sum\limits_{T\in\mathcal{T}_{h}}h_{T}^{2}||{\bf
u}_{h}||_{1,T}^{2}\}^{1/2}||{\overline{K^{-1}}}^{1/2}(\tilde{\bf
u}_{h}-\Pi_{0}{\bf u}_{h})||.
\end{array}
\end{equation}
Due to $K^{-1}\in W_{\infty}^{1}(\mathcal{T}_{h})$, it holds
\begin{equation}\label{equation32}
((K^{-1}-\overline{K^{-1}}){\bf u}_{h},\tilde{\bf u}_{h}-\Pi_{0}{\bf
u}_{h})\lesssim||h{\bf
u}_{h}||\ ||{\overline{K^{-1}}}^{1/2}(\tilde{\bf u}_{h}-\Pi_{0}{\bf
u}_{h})||.
\end{equation}
In view of  the approximation property (\ref{pressureadd}) of  $\Pi_{0}$, we have
\begin{equation}\label{equation33}
\begin{array}{lll}
& &(\overline{K^{-1}}({\bf u}_{h}-\Pi_{0}{\bf u}_{h}),\tilde{\bf
u}_{h}-\Pi_{0}{\bf u}_{h})
\displaystyle
%\sum\limits_{T\in\mathcal{T}_{h}}||\overline{K^{-1}}({\bf
%u}_{h}-\Pi_{0}{\bf u}_{h})||_{T}||\tilde{\bf u}_{h}-\Pi_{0}{\bf
%u}_{h}||_{T}\vspace{2mm}\\
%& &\hspace{1.5cm} \ \ \ \
%\displaystyle\lesssim\sum\limits_{T\in\mathcal{T}_{h}} h_{T}||{\bf
%u}_{h}||_{1,T}||\tilde{\bf u}_{h}-\Pi_{0}{\bf
%u}_{h}||_{T}\vspace{2mm}\\
%& &\hspace{1.5cm} \ \ \ \
\displaystyle\lesssim(\sum\limits_{T\in\mathcal{T}_{h}}h_{T}^{2}||{\bf
u}_{h}||_{1,T}^{2})^{1/2}||{\overline{K^{-1}}}^{1/2}(\tilde{\bf
u}_{h}-\Pi_{0}{\bf u}_{h})||.
\end{array}
\end{equation}
Combining (\ref{equation30})-(\ref{equation33}) leads to the desired estimate (\ref{equation27+2}).
\end{proof}

We now follow the idea of \cite{Vohralik1} to construct a postprocessed scalar pressure $l_{h}$ which links
$\tilde{\bf u}_{h}$ and $\tilde{p}_{h}$ on each simplicial element
in the following way:
\begin{equation}\label{equation34}
-\overline{K^{-1}}^{-1}\nabla l_{h}=\tilde{\bf u}_{h}\ \ \ {\rm in}\
\ \ T,\ \ \ {\rm for\ all}\ \ \ T\in\mathcal{T}_{h},
\end{equation}
\begin{equation}\label{equation35}
\displaystyle\frac{1}{|T|}\int_{T}l_{h}d{\bf x}=\tilde{p}_{h}|_{T},\
\ \ {\rm for\ all}\ \ \ T\in\mathcal{T}_{h}.
\end{equation}
We refer to   \cite{Vohralik1} for the existence of the
postprocessed solution $l_{h}$.

As shown in \cite{Vohralik1},   the new quantity $l_{h}$ has the continuity of the mean values of
traces across interior sides (or faces), and its mean of trace on any boundary side (or face) equals to that of
$g$. In fact, for an interior side  (or face)
$E$ shared by $T_{+}$ and $T_{-}$, let ${\bf v}_{E}$ denote the side (or face)
basis function on $E$ with respect to ${\bf V}_{h}^{0}$ with the
support set $\omega_{E}$. From  (\ref{equation26}),
 (\ref{equation34})-(\ref{equation35}) and 
  integration by parts we have
\begin{equation*}
\begin{array}{lll}
0&=&\displaystyle(-\nabla_{h}l_{h},{\bf v}_{E})_{T_{+}\cup
T_{-}}-(\tilde{p}_{h},\nabla\cdot{\bf v}_{E})_{T_{+}\cup
T_{-}}+<g,{\bf
v}_{E}\cdot{\bf n}>_{\partial\omega_{E}\cap\Gamma_{D}}\vspace{2mm}\\
&=&\displaystyle\int_{T_{+}}\nabla\cdot{\bf
v}_{E}(l_{h}-\tilde{p}_{h})+\int_{T_{-}}\nabla\cdot{\bf v}_{E}(l_{h}-\tilde{p}_{h})+\int_{E}{\bf
v}_{E}\cdot{\bf n}_{E}(l_{h}|_{T_{+}}-l_{h}|_{T_{-}})\vspace{2mm}\\
%& &\ \displaystyle-\int_{T_{+}}\tilde{p}_{h}\nabla\cdot{\bf
%v}_{E}-\int_{T_{-}}\tilde{p}_{h}\nabla\cdot{\bf v}_{E}\vspace{2mm}\\
&=&\displaystyle<1,l_{h}|_{T_{+}}-l_{h}|_{T_{-}}>_{E},
\end{array}
\end{equation*}
which implies the continuity of the means of traces of $l_{h}$
across the interior side. For a boundary side $E\subset\Gamma_{D}$,
let $E\subset\partial T$. Similarly, from (\ref{equation26}) and
(\ref{equation34})-(\ref{equation35}) we have
\begin{eqnarray*}
0&=&-(\nabla l_{h},{\bf v}_{E})_{T}-(\tilde{p}_{h},\nabla\cdot{\bf
v}_{E})_{T}+<g,{\bf v}_{E}\cdot{\bf n}>_{\partial T\cap\Gamma_{D}}\\
&=&<1,g-l_{h}>_{E}.
\end{eqnarray*}

For $K^{-1}\in W_{\infty}^{1}(\mathcal{T}_{h}),$ from  the  triangle inequality,  the postprocessing (\ref{equation34}),  an interpolation estimate, an inverse inequality,    Lemma
\ref{equation27+1} and the definition (\ref{quadrature-error}) of the quadrature indicator
$\eta_{Q}$  it follows
\begin{equation}\label{equation36}
\begin{array}{lll}
\inf\limits_{v_{h}\in\tilde{S}} ||\nabla v_{h}-K^{-1}{\bf
u}_{h}||&\leq&\inf\limits_{v_{h}\in\tilde{S}}\left\{\displaystyle||\nabla
v_{h}-\overline{K^{-1}}\tilde{\bf
u}_{h}||+||\overline{K^{-1}}\tilde{\bf
u}_{h}-\overline{K^{-1}}\Pi_{0}{\bf u}_{h}||\right.\vspace{2mm}\\
& &\ \ \left.\displaystyle+||\overline{K^{-1}}\Pi_{0}{\bf
u}_{h}-\overline{K^{-1}}{\bf u}_{h}||+||\overline{K^{-1}}{\bf
u}_{h}-K^{-1}{\bf u}_{h}||\right\}\vspace{2mm}\\
&\lesssim&\inf\limits_{v_{h}\in\tilde{S}}\displaystyle\left\{||\nabla_{h}(v_{h}+l_{h})||+||\overline{K^{-1}}^{1/2}(\tilde{\bf
u}_{h}-\Pi_{0}{\bf u}_{h})||\right.\vspace{2mm}\\
& &\ \ \left.
\displaystyle+(\sum\limits_{T\in\mathcal{T}_{h}}h_{T}^{2}||{\bf
u}_{h}||_{1,T}^{2})^{1/2}+||h{\bf u}_{h}||\right\}\vspace{2mm}\\
&\lesssim&
\inf\limits_{v_{h}\in\tilde{S}}||h^{-1}(v_{h}+l_{h})||+\eta_{Q}.
\end{array}
\end{equation}
%Using the inverse estimate, we have
%\begin{equation}\label{equation37}
%||\nabla_{h}(v_{h}+l_{h})||\lesssim||h^{-1}(v_{h}+l_{h})||.
%\end{equation}
%
%From (\ref{equation36})-(\ref{equation37}), altogether, we arrive at
%\begin{equation*}
%||\nabla v_{h}-K^{-1}{\bf u}_{h}||\lesssim
%||h^{-1}(v_{h}+l_{h})||+\eta_{Q}.
%\end{equation*}
%The above inequality leads to
%\begin{equation}\label{equation40}
%\inf_{v_{h}\in\tilde{S}}||\nabla v_{h}-K^{-1}{\bf u}_{h}||\lesssim
%\inf_{v_{h}\in\tilde{S}}||h^{-1}(v_{h}+l_{h})||+\eta_{Q}.
%\end{equation}

 Following the idea of the
proof of Lemma 3.4 in \cite{Carstensen;Bartels}, we easily obtain the following conclusion.
\begin{lemma}\label{equation41}
Let $l_{h}$ be the postprocessed scalar pressure determined by
(\ref{equation34})-(\ref{equation35}), and $g_{h,D}$ be the
nodal $\varepsilon_{D}-$piecewise linear interpolation of $g$ on
$\Gamma_{D}$. For a side (or face) $E\in\varepsilon_{h}$, denote
\begin{equation*}
\tilde{J}_{{\bf t}_{E}} :=\left \{ \begin{array}{ll}
h_{E}^{1/2}||[l_{h}]||_{E}, &   \mbox{if}\ \ E\in\varepsilon_{h}^{0},\\
h_{E}^{1/2}||l_{h}-g_{h,D}||_{E}, &   \mbox{if}\ \
E\in\varepsilon_{D}.
 \end{array}\right.
\end{equation*}
Then it holds
\begin{equation}\label{equation42}
\inf_{v_{h}\in\tilde{S}}||h^{-1}(v_{h}+l_{h})||^{2}\lesssim
\displaystyle\sum\limits_{E\in\varepsilon_{h}^{0}\cup\varepsilon_{D}}h_{E}^{-2}\tilde{J}_{{\bf
t}_{E}}^{2}.
\end{equation}
\end{lemma}
%\begin{proof}
%We get the estimate (\ref{equation42})
%\end{proof}

Using Lemma \ref{equation41}, we have a further conclusion as follows.
\begin{lemma}\label{equation43}
Let $J_{{\bf t}_{E}}$ and $\eta_{Q}$ denote the tangential jump and the quadrature indicator defined in (\ref{J-tE}) and   (\ref{quadrature-error}), respectively. Under the assumption of Lemma \ref{equation27+1},  it holds 
\begin{equation}\label{equation44}
\inf_{v_{h}\in\tilde{S}}||h^{-1}(v_{h}+l_{h})||^{2}\lesssim\displaystyle
\sum\limits_{E\in\varepsilon_{h}^{0}\cup\varepsilon_{D}}h_{E}J_{{\bf
t}_{E}}^{2}+\eta_{Q}^{2}.
\end{equation}
\end{lemma}
\begin{proof}
We only prove the three-dimensional case, since the two-dimensional
one is somewhat simpler and can be derived similarly. In the case
$E=T_{+}\cap T_{-}\in\varepsilon_{h}^{0}$, since
$\displaystyle\int_{E}[l_{h}]ds$ vanishes, a sidewise Poincar\'{e}
inequality and the postprocessing (\ref{equation34}) yield that
\begin{equation}\label{equation45}
\begin{array}{lll}
||[l_{h}]||_{E}&\lesssim&h_{E}||(\nabla l_{h}|_{T_{+}}-\nabla
l_{h}|_{T_{-}})\times{\bf n}_{E}||_{E}\vspace{2mm}\\
&=&h_{E}||(\overline{K^{-1}}\tilde{\bf
u}_{h}|_{T_{-}}-\overline{K^{-1}}\tilde{\bf
u}_{h}|_{T_{+}})\times{\bf n}_{E}||_{E}.
\end{array}
\end{equation}

Recall that $\Pi_{0}$ is the projection from ${\bf
H}({\rm div};\Omega)$ onto ${\bf V}_{h}^{0}$, and notice that
\begin{equation}\label{equation46}
\begin{array}{lll}
& &\overline{K^{-1}}\tilde{\bf
u}_{h}|_{T_{-}}-\overline{K^{-1}}\tilde{\bf
u}_{h}|_{T_{+}}\vspace{2mm}\\
&=&(\overline{K^{-1}}\tilde{\bf
u}_{h}|_{T_{-}}-\overline{K^{-1}}\Pi_{0}{\bf
u}_{h}|_{T_{-}})+(\overline{K^{-1}}\Pi_{0}{\bf
u}_{h}|_{T_{+}}-\overline{K^{-1}}\tilde{\bf u}_{h}|_{T_{+}})\vspace{2mm}\\
&&+(\overline{K^{-1}}\Pi_{0}{\bf
u}_{h}|_{T_{-}}-\overline{K^{-1}}\Pi_{0}{\bf
u}_{h}|_{T_{+}})\vspace{2mm}\\
&=&(\overline{K^{-1}}\tilde{\bf
u}_{h}|_{T_{-}}-\overline{K^{-1}}\Pi_{0}{\bf
u}_{h}|_{T_{-}})+(\overline{K^{-1}}\Pi_{0}{\bf
u}_{h}|_{T_{+}}-\overline{K^{-1}}\tilde{\bf u}_{h}|_{T_{+}})\vspace{2mm}\\
&&+(\overline{K^{-1}}\Pi_{0}{\bf
u}_{h}|_{T_{-}}-K^{-1}\Pi_{0}{\bf
u}_{h}|_{T_{-}})+(K^{-1}\Pi_{0}{\bf u}_{h}|_{T_{-}}-K^{-1}{\bf
u}_{h}|_{T_{-}})\vspace{2mm}\\
&&+(K^{-1}{\bf u}_{h}|_{T_{-}}-K^{-1}{\bf
u}_{h}|_{T_{+}})
+(K^{-1}{\bf u}_{h}|_{T_{+}}-K^{-1}\Pi_{0}{\bf
u}_{h}|_{T_{+}})\vspace{2mm}\\
&&+(K^{-1}\Pi_{0}{\bf
u}_{h}|_{T_{+}}-\overline{K^{-1}}\Pi_{0}{\bf u}_{h}|_{T_{+}}).
\end{array}
\end{equation}
%where
%\begin{equation}\label{equation47}
%\begin{array}{lll}
%& &\overline{K^{-1}}\Pi_{0}{\bf
%u}_{h}|_{T_{-}}-\overline{K^{-1}}\Pi_{0}{\bf
%u}_{h}|_{T_{+}}=(\overline{K^{-1}}\Pi_{0}{\bf
%u}_{h}|_{T_{-}}-K^{-1}\Pi_{0}{\bf
%u}_{h}|_{T_{-}})\vspace{2mm}\\
%& &\ \ \quad \quad+(K^{-1}\Pi_{0}{\bf u}_{h}|_{T_{-}}-K^{-1}{\bf
%u}_{h}|_{T_{-}})+(K^{-1}{\bf u}_{h}|_{T_{-}}-K^{-1}{\bf
%u}_{h}|_{T_{+}})\vspace{2mm}\\
%& &\ \ \quad\quad +(K^{-1}{\bf u}_{h}|_{T_{+}}-K^{-1}\Pi_{0}{\bf
%u}_{h}|_{T_{+}})+(K^{-1}\Pi_{0}{\bf
%u}_{h}|_{T_{+}}-\overline{K^{-1}}\Pi_{0}{\bf u}_{h}|_{T_{+}}).
%\end{array}
%\end{equation}

Employing the trace theorem, inverse estimate and  the
local shape regularity of the mesh, we have
\begin{equation}\label{equation48}
\begin{array}{lll}
& &||(\overline{K^{-1}}\tilde{\bf
u}_{h}|_{T_{-}}-\overline{K^{-1}}\Pi_{0}{\bf
u}_{h}|_{T_{-}})\times{\bf
n}_{E}||_{E}+||(\overline{K^{-1}}\Pi_{0}{\bf
u}_{h}|_{T_{+}}-\overline{K^{-1}}\tilde{\bf
u}_{h}|_{T_{+}})\times{\bf n}_{E}||_{E}\vspace{2mm}\\
& \lesssim&  h_{E}^{-1/2}||\overline{K^{-1}}(\tilde{\bf
u}_{h}-\Pi_{0}{\bf u}_{h})||_{\omega_{E}}.
\end{array}
\end{equation}
The trace theorem, together with the stable estimate
(\ref{equation12+3}) on the operator $\Pi_{0}$, also indicates
\begin{equation}\label{equation49}
\begin{array}{lll}
& &||(\overline{K^{-1}}\Pi_{0}{\bf u}_{h}|_{T_{-}}-K^{-1}\Pi_{0}{\bf
u}_{h}|_{T_{-}})\times{\bf n}_{E}||_{E}\vspace{2mm}\\
&\leq & ||(\overline{K^{-1}}-K^{-1})\Pi_{0}{\bf
u}_{h}|_{T_{-}}||_{\partial T_{-}}\vspace{2mm}\\
&\lesssim & ||(\overline{K^{-1}}-K^{-1})\Pi_{0}{\bf
u}_{h}||_{T_{-}}^{1/2}||(\overline{K^{-1}}-K^{-1})\Pi_{0}{\bf
u}_{h}||_{1,T_{-}}^{1/2}\vspace{2mm}\\
%& &\ \ \lesssim
%h_{T_{-}}^{1/2}||K^{-1}||_{1,\infty,T_{-}}^{1/2}||\Pi_{0}{\bf
%u}_{h}||_{T_{-}}^{1/2}||\overline{K^{-1}}-K^{-1}||_{1,\infty,T_{-}}^{1/2}||\Pi_{0}{\bf
%u}_{h}||_{1,T_{-}}^{1/2}\vspace{2mm}\\
& \lesssim & h_{T_{-}}^{1/2}||{\bf u}_{h}||_{1,T_{-}}.
\end{array}
\end{equation}
Similarly, it holds
\begin{equation}\label{equation50}
||(K^{-1}\Pi_{0}{\bf u}_{h}|_{T_{+}}-\overline{K^{-1}}\Pi_{0}{\bf
u}_{h}|_{T_{+}})\times{\bf n}_{E}||_{E}\lesssim
h_{T_{+}}^{1/2}||{\bf u}_{h}||_{1,T_{+}},
\end{equation}
%A combination of (\ref{equation49}) and (\ref{equation50}) yields
%\begin{equation}\label{equation51}
%\begin{array}{lll}
%&\ &||(\overline{K^{-1}}-K^{-1})\Pi_{0}{\bf
%u}_{h}|_{T_{-}}\times{\bf
%n}_{E}||_{E}+||(K^{-1}-\overline{K^{-1}})\Pi_{0}{\bf
%u}_{h}|_{T_{+}}\times{\bf n}_{E}||_{E}\vspace{2mm}\\
%&\ &\hspace{5mm}\lesssim h_{E}^{1/2}||{\bf u}_{h}||_{1,\omega_{E}}.
%\end{array}
%\end{equation}
\begin{equation}\label{equation51}
||(K^{-1}\Pi_{0}{\bf u}_{h}|_{T_{-}}-K^{-1}{\bf
u}_{h}|_{T_{-}})\times{\bf n}_{E}||_{E}\lesssim h_{E}^{1/2}||{\bf u}_{h}||_{1,T_{-}},
\end{equation}
and
\begin{equation}\label{equation52}
||(K^{-1}{\bf
u}_{h}|_{T_{+}}-K^{-1}\Pi_{0}{\bf u}_{h}|_{T_{+}})\times{\bf
n}_{E}||_{E}\lesssim h_{E}^{1/2}||{\bf u}_{h}||_{1,T_{+}},
\end{equation}
where in the latter two inequalities we have also used  the estimate (\ref{pressureadd}).

As a result, a combination of  (\ref{equation45})-(\ref{equation52}) shows
\begin{equation}\label{equation53}
\begin{array}{lll}
||[l_{h}]||_{E}&\lesssim&h_{E}\{h_{E}^{-1/2}||\overline{K^{-1}}(\tilde{\bf
u}_{h}-\Pi_{0}{\bf u}_{h})||_{\omega_{E}}\vspace{2mm}\\
& &\ \ \ +h_{E}^{1/2}||{\bf u}_{h}||_{1,\omega_{E}}+||[\gamma_{{\bf
t}_{E}}(K^{-1}{\bf u}_{h})]||_{E}\}.
\end{array}
\end{equation}

On the other hand, in the case $E\subset\partial T\cap\varepsilon_{D}$ it holds
\begin{equation*}
\displaystyle\frac{1}{|E|}\int_{E}(l_{h}-g)ds=0
\end{equation*}
due to
$\displaystyle\int_{E}l_{h}ds=\int_{E}gds$.
Using the triangle inequality, sidewise Poincar\'{e} inequality and
interpolation estimation, we have
\begin{equation}\label{equation54}
\begin{array}{lll}
||l_{h}-g_{h,D}||_{E}&\leq&||l_{h}-g||_{E}+||g-g_{h,D}||_{E}\vspace{2mm}\\
&\lesssim&h_{E}||\nabla l_{h}\times{\bf n}_{E}-\partial g/\partial
s||_{E}+h_{E}^{2}||\partial^{2}g/\partial s^{2}||_{E}.
\end{array}
\end{equation}
Similarly it holds
\begin{equation}\label{equation55}
\begin{array}{lll}
h_{E}||\nabla l_{h}\times{\bf n}_{E}-\frac{\partial g}{\partial
s}||_{E}&\lesssim& h_{E}^{1/2}||\overline{K^{-1}}(\tilde{\bf
u}_{h}-\Pi_{0}{\bf u}_{h})||_{T}+h_{E}^{3/2}||{\bf
u}_{h}||_{1,T}\vspace{2mm}\\
& &\ +h_{E}||K^{-1}{\bf u}_{h}\times{\bf n}_{E}-\partial g/\partial
s||_{E}.
\end{array}
\end{equation}
The above two estimates, (\ref{equation54}) and (\ref{equation55}), lead to
\begin{equation}\label{equation56}
\begin{array}{lll}
||l_{h}-g_{h,D}||_{E}&\lesssim&h_{E}^{1/2}||\overline{K^{-1}}(\tilde{\bf
u}_{h}-\Pi_{0}{\bf u}_{h})||_{T}+h_{E}^{3/2}||{\bf
u}_{h}||_{1,T}\vspace{2mm}\\
& &\ +h_{E}||K^{-1}{\bf u}_{h}\times{\bf n}_{E}-\partial g/\partial
s||_{E}+h_{E}^{2}||\partial^{2}g/\partial s^{2}||_{E}.
\end{array}
\end{equation}
%From the definition of $\tilde{J}_{{\bf t}_{E}}$ in Lemma
%\ref{equation41}, summing (\ref{equation53}) on all interior side
%$E\in\varepsilon_{h}^{0}$, and summing (\ref{equation56}) on all
%Dirichlet boundary side $E\subset\varepsilon_{D}$, and noticing that
%an element overlaps at most $d+1$ times in summation, we have
From the definition of $\tilde{J}_{{\bf t}_{E}}$ in Lemma
\ref{equation41}, the estimates (\ref{equation53}) and (\ref{equation56}) indicate
\begin{equation}\label{equation57}
\begin{array}{lll}
&
&\displaystyle\sum\limits_{E\in\varepsilon_{h}^{0}\cup\varepsilon_{D}}h_{E}^{-2}\tilde{J}_{{\bf
t}_{E}}^{2}=
\displaystyle\sum\limits_{E\in\varepsilon_{h}^{0}}h_{E}^{-2}h_{E}||[l_{h}]||_{E}^{2}+
\sum\limits_{E\in\varepsilon_{D}}h_{E}^{-2}h_{E}||l_{h}-g_{h,D}||_{E}^{2}\vspace{2mm}\\
&\lesssim &
\displaystyle\sum\limits_{T\in\mathcal{T}_{h}}h_{T}^{2}||{\bf
u}_{h}||_{1,T}^{2}+||\overline{K^{-1}}(\tilde{\bf u}_{h}-\Pi_{0}{\bf
u}_{h})||^{2}+\sum\limits_{E\in\varepsilon_{h}^{0}}h_{E}||[\gamma_{{\bf
t}_{E}}(K^{-1}{\bf
u}_{h})]||_{E}^{2}\vspace{2mm}\\
& &\ \ \ \
\displaystyle+\sum\limits_{E\in\varepsilon_{D}}(h_{E}||\gamma_{{\bf
t}_{E}}(K^{-1}{\bf u}_{h})-\partial g/\partial
s||_{E}^{2}+h_{E}^{3}||\partial^{2}g/\partial s^{2}||_{E}).
\end{array}
\end{equation}
By noticing that Lemma \ref{equation27+1} implies
\begin{equation*}%\label{equation58}
||\overline{K^{-1}}(\tilde{\bf u}_{h}-\Pi_{0}{\bf
u}_{h})||^{2}\lesssim\eta_{Q}^{2},
\end{equation*}
the estimate (\ref{equation57}), together with the
definitions of $J_{{\bf t}_{E}}$ and  $\eta_{Q}$, (\ref{J-tE}) and (\ref{eta-Q}),   yields
\begin{equation}\label{equation60}
\sum\limits_{E\in\varepsilon_{h}^{0}\cup\varepsilon_{D}}h_{E}^{-2}\tilde{J}_{{\bf
t}_E}^{2}\lesssim
\sum\limits_{E\in\varepsilon_{h}^{0}\cup\varepsilon_{D}}h_{E}J_{{\bf
t}_{E}}^{2}+\eta_{Q}^{2}.
\end{equation}
The desired result (\ref{equation44}) follows from Lemma
\ref{equation41} and (\ref{equation60}).
\end{proof}

{\bf The proof of Theorem \ref{equation7}}: Collecting
(\ref{equation25}), (\ref{equation36}) and (\ref{equation44}), we get
\begin{equation}\label{equation61}
\inf_{\beta\in H^{1}(\Omega)}||\nabla\beta-K^{-1}{\bf u}_{h}||\lesssim\{\sum\limits_{E\in\varepsilon_{h}^{0}\cup\varepsilon_{D}}h_{E}J_{{\bf
t}_{E}}^{2}\}^{1/2}+ \eta_{Q},
\end{equation}
which, together with the estimates  (\ref{equation22})-(\ref{equation24}), yields
\begin{equation}\label{equation62}
\begin{array}{lll}
& &\displaystyle||K^{-1/2}({\bf u}-{\bf u}_{h})||\lesssim
||h(f-\nabla\cdot{\bf
u}_{h})||+\{\sum\limits_{E\in\varepsilon_{D}}h_{E}^{3}||\partial^{2}g/\partial
s^{2}||_{E}^{2}\}^{1/2}\vspace{2mm}\\
& &\ \ \
\displaystyle\hspace{3.0cm}+\{\sum\limits_{E\in\varepsilon_{h}^{0}\cup\varepsilon_{D}}h_{E}J_{{\bf
t}_{E}}^{2}\}^{1/2}+
\eta_{Q}\vspace{2mm}\\
& &\ \displaystyle\lesssim\{||h(f-\nabla\cdot{\bf
u}_{h})||+\sum\limits_{E\in\varepsilon_{h}^{0}\cup\varepsilon_{D}}h_{E}J_{{\bf
t}_{E}}^{2}\}^{1/2}+ \eta_{Q}.
\end{array}
\end{equation}
The desired result (\ref{equation8}) then follows from (\ref{equation62})
and the definition (\ref{eta-h}) of $\eta_{h}$.

{\bf The proof of Theorem \ref{newpressure}}: Recall that $Q_{h}$ is
the $L^{2}-$projection operator onto $W_{h}$. Construct the following
auxiliary problem: Find $\phi\in H^{1}(\Omega)$ such that
\begin{equation}\label{pressure1}
\left \{ \begin{array}{ll}
  \nabla\cdot(K\nabla \phi)=Q_{h}p-p_{h} &\quad   \mbox{in}\ \ \Omega,\\
 \hspace{14.4mm}\phi=0 &\quad  \mbox{on}\ \partial\Omega.
 \end{array}\right.
\end{equation}
By the assumptions of $K$ and Lax-Milgram theorem, the operator
\begin{equation*}
\nabla\cdot(K\nabla\cdot):H_{0}^{1}(\Omega)\rightarrow
H^{-1}(\Omega)
\end{equation*}
is invertible and it holds the following regularity estimate:
\begin{equation}\label{pressure2}
||\phi||_{1}\lesssim||Q_{h}p-p_{h}||.
\end{equation}
Moreover, if $\Omega$ is convex,
$K\in\mathcal{C}^{1,0}(\overline{\Omega})$ implies that
\begin{equation*}
\nabla\cdot(K\nabla\cdot):H_{0}^{1}(\Omega)\cap
H^{2}(\Omega)\rightarrow L^{2}(\Omega)
\end{equation*}
is invertible (\cite{Grisvard}) and the   regularity estimate
\begin{equation}\label{dual problem 1}
||\phi||_{H^{2}(\bigcup\mathcal{T}_{h})}\lesssim||Q_{h}p-p_{h}||
\end{equation}
holds.
We emphasize that here we only need a regularity estimate on $||\phi||_{H^{2}(T)}$
for each $T\in\mathcal{T}_{h}$ and then assume a weakened constraint on $K$ such that
 (\ref{dual problem 1}) holds. In
\cite{Carstensen0}  {\it Carstensen} gave an example  where
 $K$ is piecewise constant and $\phi$ satisfies (\ref{dual
problem 1}) but is not $H^{2}$-regular.

Notice that the error equation of the MFMFE method (\ref{equation4})-(\ref{equation5}) can be written as
\begin{equation}\label{pressure3}
(K^{-1}({\bf u}-{\bf u}_{h}),{\bf
v}_{h})=(Q_{h}p-p_{h},\nabla\cdot{\bf v}_{h})-\sigma(K^{-1}{\bf
u}_{h},{\bf v}_{h}),\ \ \forall{\bf v}_{h}\in V_{h}.
\end{equation}
Recalling $\Pi $ is the standard projection operator  from ${\bf H}({\rm div};\Omega)\cap
(L^{\varrho}(\Omega))^{d}$ onto $V_{h}$, and taking ${\bf v}_{h}=\Pi (K\nabla\phi)$ in (\ref{pressure3}), from (\ref{pressure1})  and
   the commuting property  $\nabla\cdot (\Pi K\nabla\phi)=Q_h\nabla\cdot ( K\nabla\phi)$, 
 we have
\begin{equation}\label{pressure4}
\begin{array}{lll}
||Q_{h}p-p_{h}||^{2}&=&(Q_{h}p-p_{h},\nabla\cdot(\Pi
K\nabla\phi))\vspace{2mm}\\
&=&(K^{-1}({\bf u}-{\bf u}_{h}),\Pi (K\nabla\phi))+\sigma(K^{-1}{\bf
u}_{h},\Pi K\nabla\phi).
\end{array}
\end{equation}

Since $(\nabla\cdot({\bf u}-{\bf
u}_{h}),w_{h})=0, \forall w_{h}\in W_{h}$, by integration by parts, the 
approximation property of $\Pi$ and  the estimates  (\ref{pressure2})-(\ref{dual problem 1}), we have
\begin{equation}\label{pressure5}
\begin{array}{lll}
& &(K^{-1}({\bf u}-{\bf u}_{h}),\Pi (K\nabla\phi))=(K^{-1}({\bf
u}-{\bf u}_{h}),\Pi (K\nabla\phi)-K\nabla\phi)+({\bf u}-{\bf
u}_{h},\nabla\phi)\vspace{2mm}\\
& &\ \ =(K^{-1}({\bf u}-{\bf u}_{h}),\Pi
(K\nabla\phi)-K\nabla\phi)-(\nabla\cdot({\bf u}-{\bf
u}_{h}),\phi)\vspace{2mm}\\
& &\ \ =(K^{-1}({\bf u}-{\bf u}_{h}),\Pi
(K\nabla\phi)-K\nabla\phi)-(\nabla\cdot({\bf u}-{\bf
u}_{h}),\phi-Q_{h}\phi)\vspace{2mm}\\
& &\ \ \lesssim \left(||hK^{-1/2}({\bf u}-{\bf
u}_{h})||+||h\nabla\cdot({\bf u}-{\bf u}_{h})||\right)||Q_{h}p-p_{h}||.
\end{array}
\end{equation}
On the other hand, a combination of (\ref{pressure7}), (\ref{equation12+3}) and
(\ref{dual problem 1}) yields
\begin{equation}\label{pressure9}
\begin{array}{lll}
|\sigma(K^{-1}{\bf u}_{h},\Pi
K\nabla\phi)|&\lesssim&\displaystyle\sum\limits_{T\in\mathcal{T}_{h}}h_{T}^{2}||{\bf
u}_{h}||_{1,T}||\Pi (K\nabla\phi)||_{1,T}\vspace{2mm}\\
&\lesssim&\displaystyle(\sum\limits_{T\in\mathcal{T}_{h}}h_{T}^{4}||{\bf
u}_{h}||_{1,T}^{2})^{1/2}||Q_{h}p-p_{h}||.
\end{array}
\end{equation}

Noticing $\nabla\cdot({\bf u}-{\bf u}_{h})=f-Q_{h}f$, from 
(\ref{pressure4})-(\ref{pressure9}) and the estimate (\ref{equation8}) of Theorem \ref{equation7} we
obtain the assertion (\ref{newpressure1}), i.e.
\begin{equation*}
||Q_{h}p-p_{h}||\lesssim h_{\rm max}(\eta_{h}+\eta_{Q})+||h(f-\nabla\cdot{\bf u}_{h})||.
\end{equation*} 
A triangle inequality, the relation ${\bf u}=-K\nabla p$ and the approximation property of $Q_h$ further
imply
\begin{equation*}
\begin{array}{lll}
||p-p_{h}||&\leq&||p-Q_{h}p||+||Q_{h}p-p_{h}||\lesssim||h\nabla
p||+||Q_{h}p-p_{h}||\vspace{2mm}\\
&\leq&||hK^{-1}({\bf u}-{\bf u}_{h})||+||hK^{-1}{\bf
u}_{h}||+||Q_{h}p-p_{h}||.
\end{array}
\end{equation*}
This inequality, together with   the estimate (\ref{newpressure1}), leads to the conclusion
(\ref{newpressure2}).

\section{Analysis for the efficiency} This section is devoted to the proof of Theorem \ref{equation9}. For the sake of
simplicity, we assume that $K^{-1}$ is a matrix of piecewise polynomial
functions. Since the two terms $||h(f-\nabla\cdot{\bf u}_{h})||$ and 
$\displaystyle\{\sum\limits_{E\in\varepsilon_{D}}h_{E}^{3}||\frac{\partial^{2}g}{\partial s^{2}}||_{E}^{2}\}^{1/2}$ 
in $\eta_{h}$ are of high order, they are directly incorporated in $h.o.t.$ as a high order term.
Using   standard analytical techniques, we easily
obtain Lemma \ref{efficiency jump}.
\begin{lemma}\label{efficiency jump}
Let $\eta_{h}$ denote the discretization indicator given by (\ref{eta-h}).
Then it holds
\begin{equation}\label{efficiency jump 1}
\displaystyle\eta_{h}\lesssim||K^{-1/2}({\bf u}-{\bf
u}_{h})||+h.o.t.
\end{equation}
\end{lemma}
\begin{lemma}\label{efficiency volumn }
Let $\eta_{Q}$ denote the quadrature indicator given by (\ref{eta-Q}).
Then it holds
\begin{equation}\label{efficiency volumn 1}
\eta_{Q}\lesssim||K^{-1/2}({\bf u}-{\bf
u}_{h})||+||h^{-1}(p-p_{h})||.
\end{equation}
\end{lemma}
\begin{proof}
An inverse inequality and the assumption (\ref{newadd}) yield
\begin{equation}\label{equation69}
||{\bf u}_{h}||_{1,T}\lesssim h_{T}^{-1}||{\bf u}_{h}||_{T}\lesssim
h_{T}^{-1}||K^{-1}{\bf u}_{h}||_{T}.
\end{equation}
For all $T\in\mathcal{T}_{h}$, let $\psi_{T}$ denote the bubble
function on $T$ with $\psi_{T}|_{\partial T}=0$ and
$0\leq\psi_{T}\leq1$. Then the two norms, $||\psi_{T}^{1/2}\cdot||_{T}$
and $||\cdot||_{T}$, are equivalent for polynomials. 
Since $\nabla p_{h}|_{T}=0$ due to $p_h\in W_h$, it then  holds
\begin{equation}\label{equation70}
\begin{array}{lll}
||K^{-1}{\bf u}_{h}||_{T}^{2}&=&||K^{-1}{\bf u}_{h}+\nabla
p_{h}||_{T}^{2}\vspace{2mm}\\
&\lesssim&||\psi_{T}^{1/2}(K^{-1}{\bf u}_{h}+\nabla
p_{h})||_{T}^{2}\vspace{2mm}\\
&=&\left(\psi_{T}K^{-1}{\bf u}_{h},K^{-1}{\bf u}_{h}+\nabla
p_{h}\right)_{T}\vspace{2mm}\\
&=&\left(\psi_{T}K^{-1}{\bf u}_{h},K^{-1}({\bf u}_{h}-{\bf
u})\right)_{T}+\left(\psi_{T}K^{-1}{\bf u}_{h},\nabla
(p_{h}-p)\right)_{T}\vspace{2mm}\\
&=&\left(\psi_{T}K^{-1}{\bf u}_{h},K^{-1}({\bf u}_{h}-{\bf
u})\right)_{T}-\left(\nabla\cdot(\psi_{T}K^{-1}{\bf u}_{h}),p_{h}-p\right)_{T}\vspace{2mm}\\
&\lesssim&||K^{-1}{\bf u}_{h}||_{T}\left(||K^{-1/2}({\bf u}-{\bf
u}_{h})||_{T}+h_{T}^{-1}||p-p_{h}||_{T}\right),
\end{array}
\end{equation}
where in the fourth and last lines we have used   the relation ${\bf u}=-K\nabla p$ and an inverse inequality, respectively.
This inequality, together with (\ref{equation69}), shows
\begin{equation*}
h_{T}||{\bf u}_{h}||_{1,T}\lesssim||K^{-1/2}({\bf u}-{\bf
u}_{h})||_{T}+h_{T}^{-1}||(p-p_{h})||_{T},
\end{equation*}
from which the desired estimate (\ref{efficiency volumn 1}) follows.
\end{proof}

{\bf The proof of Theorem \ref{equation9}}. From (\ref{equation70})
we obtain
\begin{equation}\label{equation71}
||hK^{-1}{\bf u}_{h}||\lesssim||hK^{-1/2}({\bf u}-{\bf
u}_{h})||+||p-p_{h}||,
\end{equation}
which, together with Lemmas \ref{efficiency jump}-\ref{efficiency volumn },
leads to the desired efficiency estimate of Theorem
\ref{equation9}.

\section{Numerical experiments}
In this section, we   use two model problems to test the performance of the developed  {\it a posteriori } error estimator for the MFMFE
method. We consider two types of meshes:  uniformly refined meshes and adaptively refined meshes. The latter type of meshes is generated by 
  a standard adaptive algorithm based on the {\it a posteriori } error estimation. In the
first example, the permeability $K$ equals to identity matrix   and
$\Omega$ is an $L$-shape domain. In the second
example,  $K$ is inhomogeneous and anisotropic. We
are thus able to study how meshes adapt to various effect from lack
of regularity of solutions to non-convexity of domains. 

\subsection*{Example 7.1} We
consider the problem (\ref{equation1}) in an $L$-shape domain
$\Omega=\{(-1,1)\times(0,1)\}\cup\{(-1,0)\times(-1,0)\}$ with Dirichlet boundary conditions  and  $K=I$ (identity matrix). The exact solution is given by
\begin{equation*}
p(\rho,\theta)=\rho^{r}\sin(r\theta),
\end{equation*}
where $\rho,\theta$ are the polar coordinates, $r$ is a
parameter. We consider  two cases
for $r$: $r=0.4$ and $r=0.1$. Some
simple calculations  show $f=0$.
\begin{figure}[htbp]
  \begin{minipage}[t]{0.5\linewidth}
    \centering
    \includegraphics[width=2.5in]{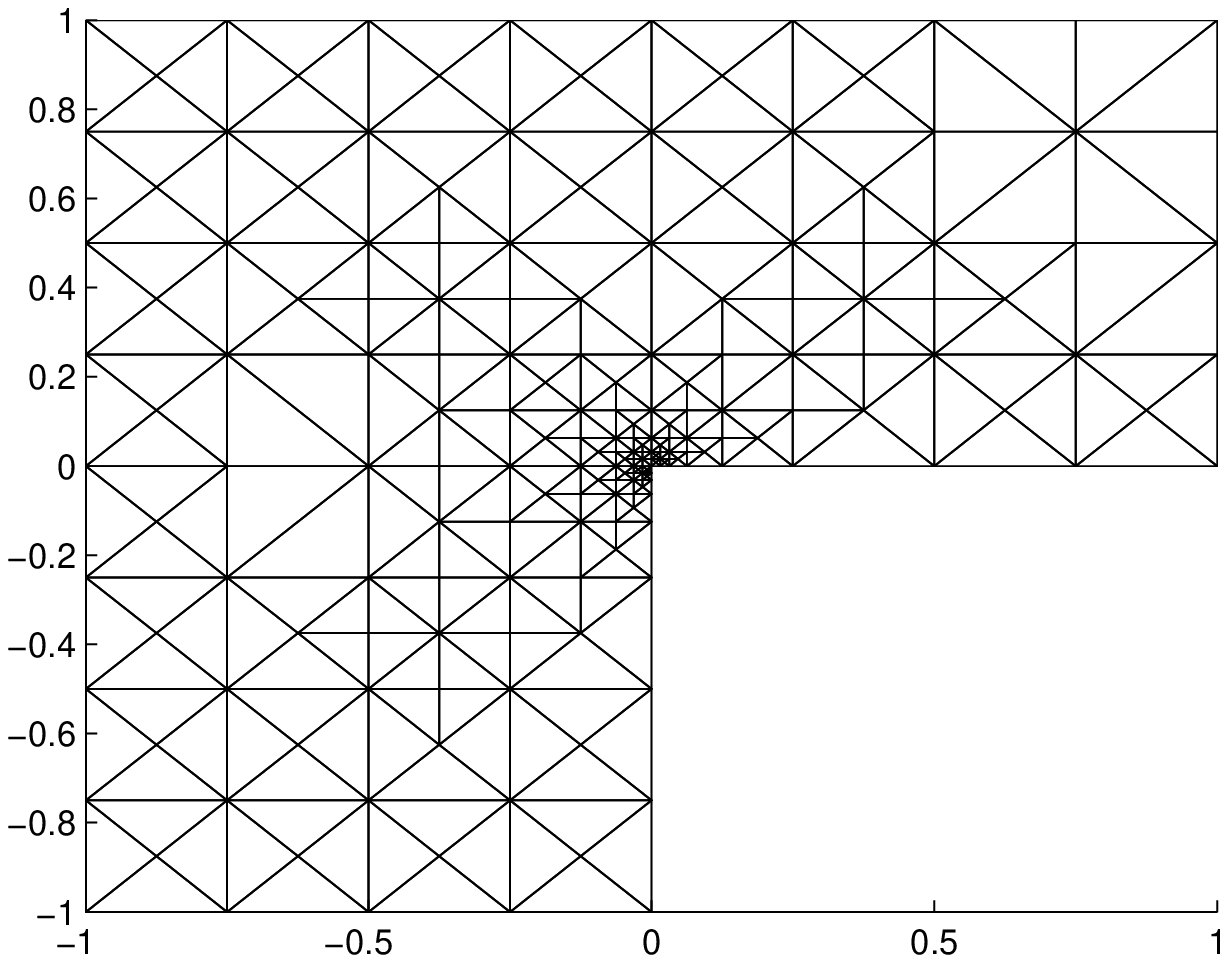}\\
  \end{minipage}
  \begin{minipage}[t]{0.5\linewidth}
    \centering
    \includegraphics[width=2.5in]{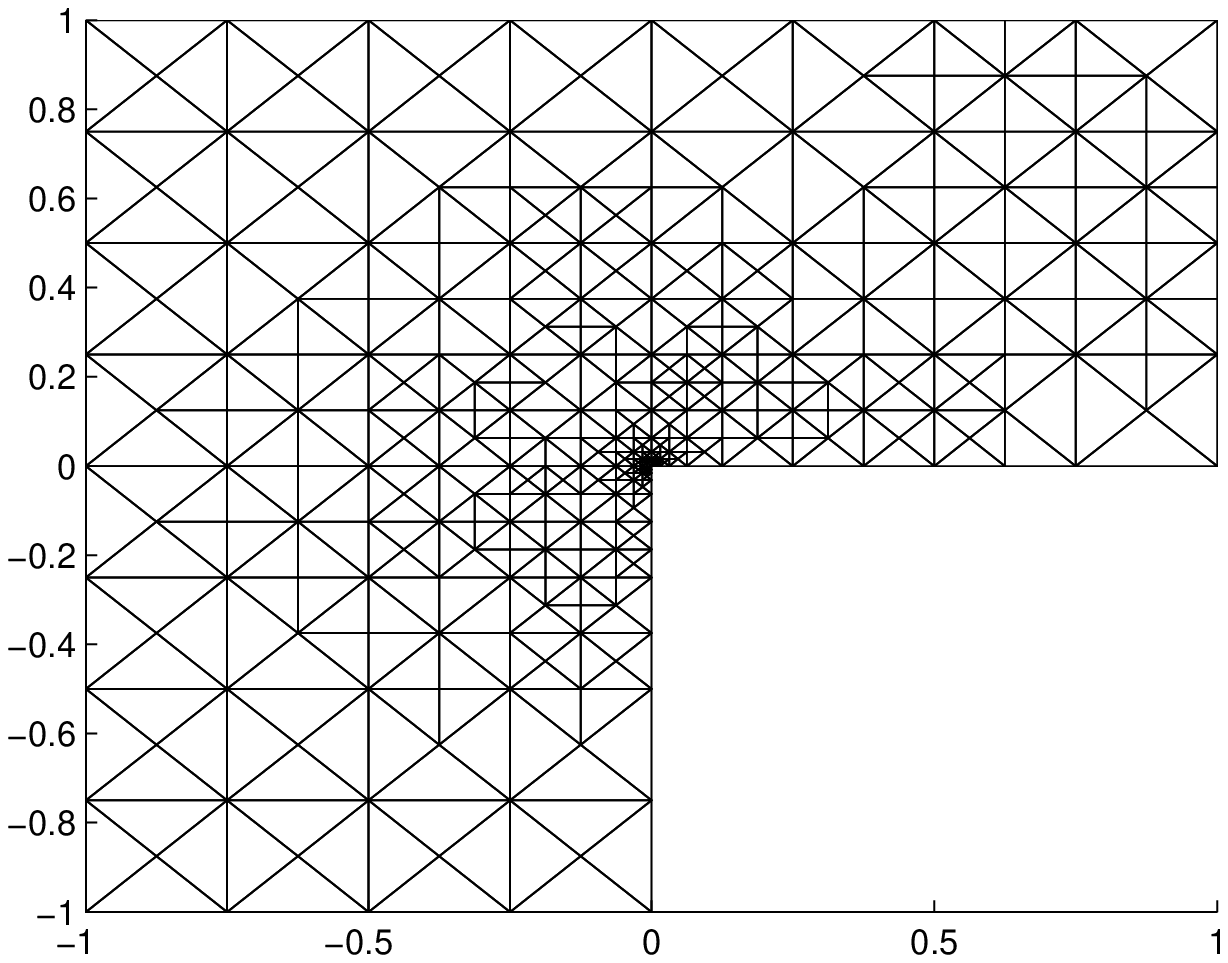}\\
  \end{minipage}
\addtocontents{lof}{figure}{FIG 7.1. {\small {\it A mesh with 347
triangles, iteration 6 (left) and a mesh with 578 triangles,
iteration 8 (right) in case $r=0.4$.}}}\\
\end{figure}
\begin{figure}[htbp]
  \begin{minipage}[t]{0.5\linewidth}
    \centering
    \includegraphics[width=2.5in]{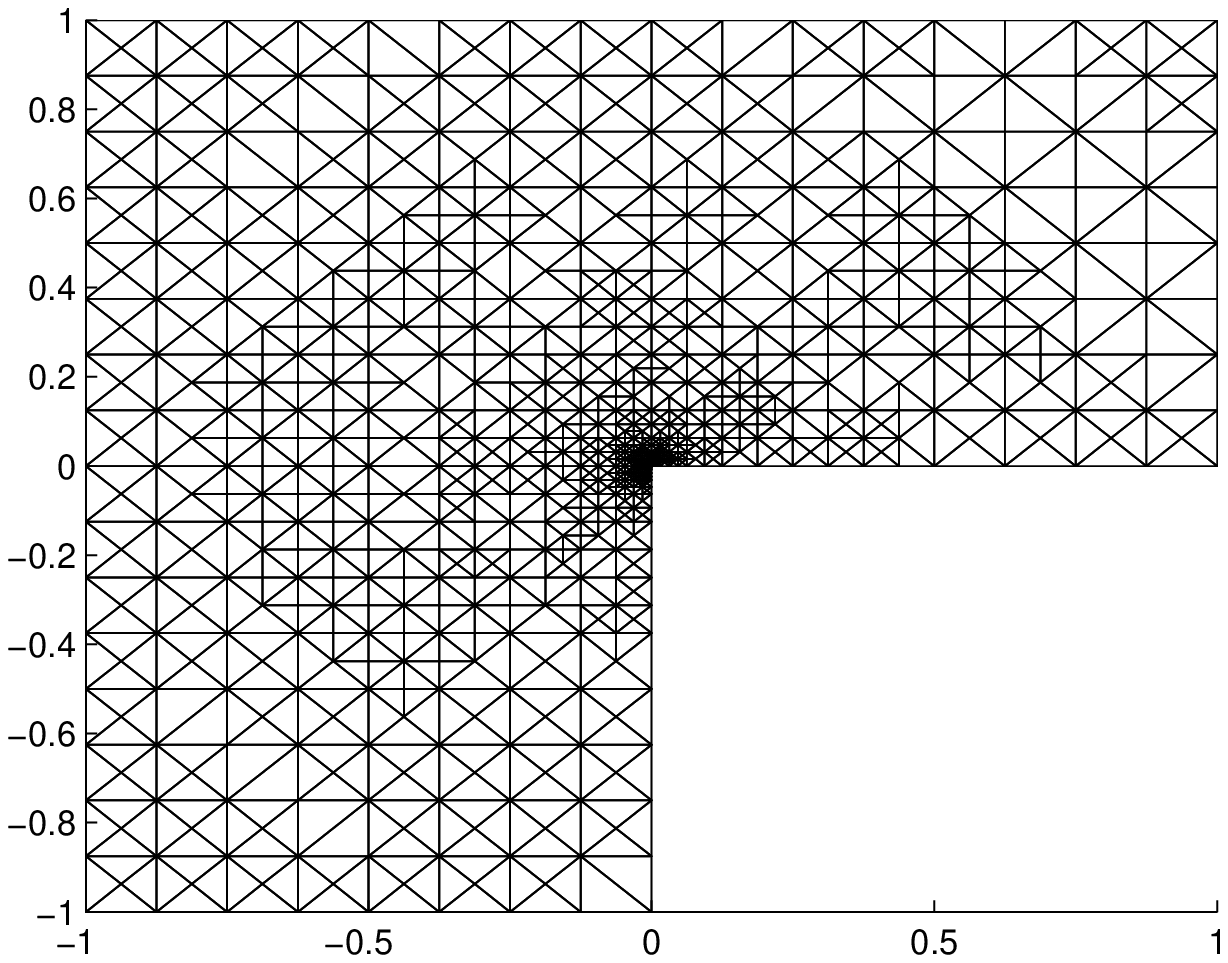}\\
  \end{minipage}
  \begin{minipage}[t]{0.5\linewidth}
    \centering
    \includegraphics[width=2.5in]{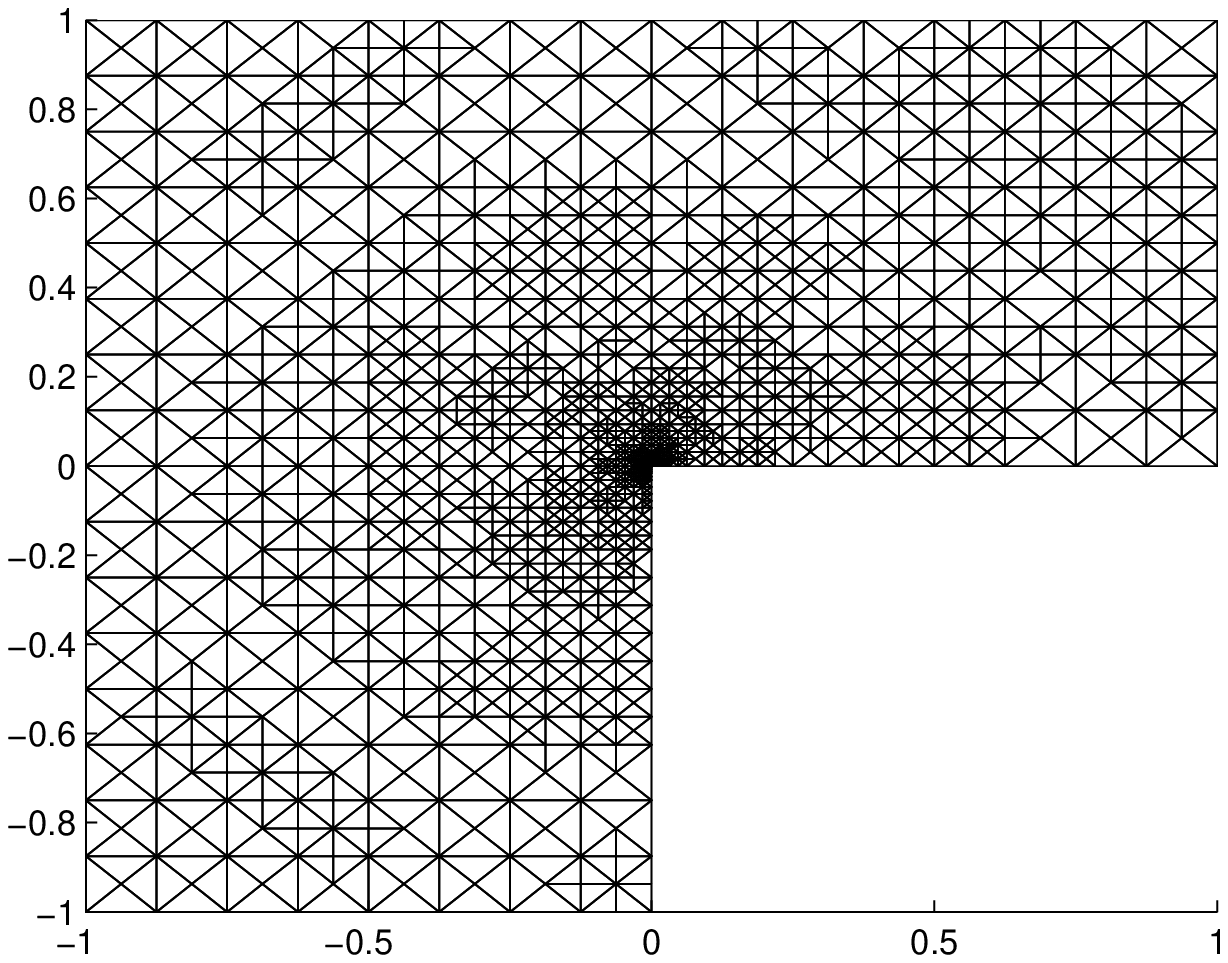}\\
  \end{minipage}
\addtocontents{lof}{figure}{FIG 7.2. {\small {\it  A mesh with 1607
triangles, iteration 11 (left) and a mesh with 2618 triangles,
iteration 12 (right) in case $r=0.4$.}}}
\end{figure}

It is well known that this model possesses singularity at the
origin and  holds $p\in H^{1+r-\epsilon}(\Omega)$ for any $\epsilon>0$. The singularity of
the solution in the case $r=0.4$ is weaker than  in the case
$r=0.1$.The original mesh consists of 6 right-angled triangles. 
%We test the {\it a posteriori } error estimators proposed in section 4. 
In the adaptive algorithm we
first solve  the
MFMFE scheme (\ref{equation4})-(\ref{equation5}), then mark elements in terms of
D\"{o}rfler marking with the marking parameter $\tilde\theta=0.5$, and finally
use the "longest edge" refinement to recover an admissible mesh.
In particular, the uniform refinement means that all elements should be
marked.
\begin{figure}[htbp]
  \begin{minipage}[t]{0.5\linewidth}
    \centering
    \includegraphics[width=2.5in]{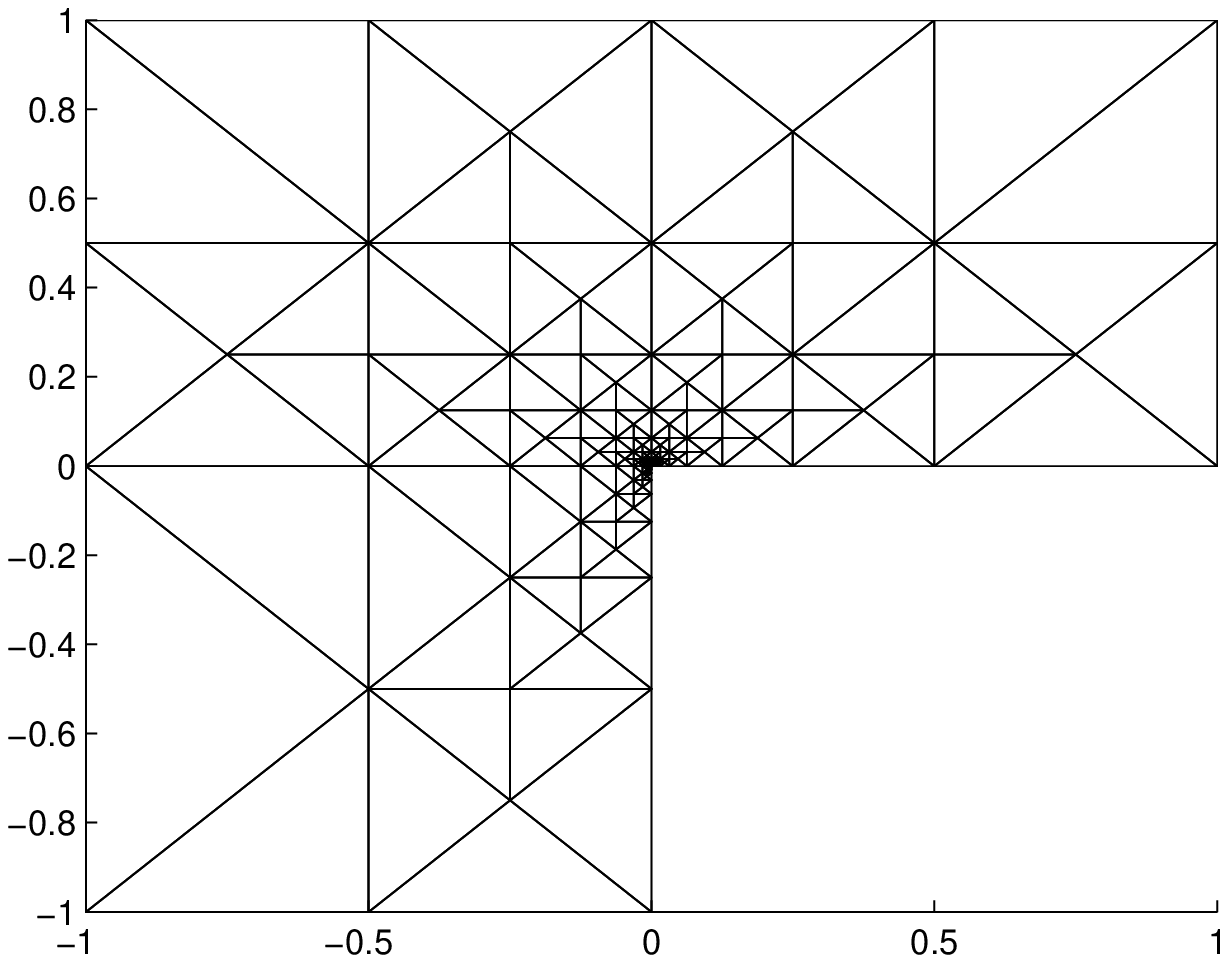}\\
  \end{minipage}
  \begin{minipage}[t]{0.5\linewidth}
    \centering
    \includegraphics[width=2.5in]{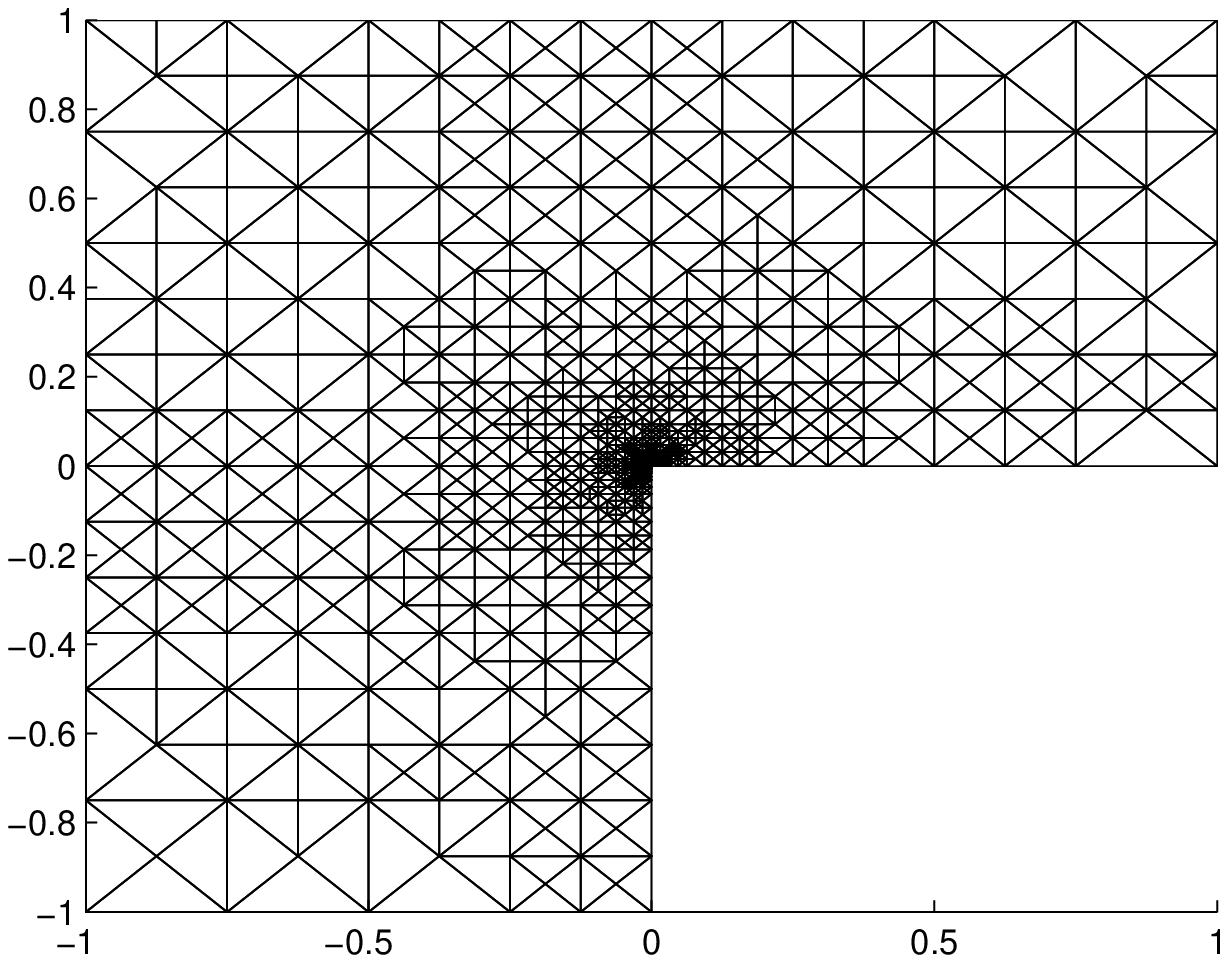}\\
  \end{minipage}
\addtocontents{lof}{figure}{FIG 7.3. {\small {\it  A mesh with 245
triangles, iteration 10 (left) and a mesh with 3265 triangles,
iteration 24 (right) in case $r=0.1$.}}}
\end{figure}

From Figs 7.1-7.2 with the   parameter $r=0.4$ and
Fig 7.3 with the parameter $r=0.1$, we see that using the adaptive algorithm the refinement
concentrates around the origin. This means that the predicted error
estimator captures well the singularity of the solution, and that
the stronger the solution possesses singularity, the better the {\it a
posteriori} error estimator can identify.
\begin{figure}[htbp]
  \begin{minipage}[t]{0.5\linewidth}
    \centering
    \includegraphics[width=2.5in]{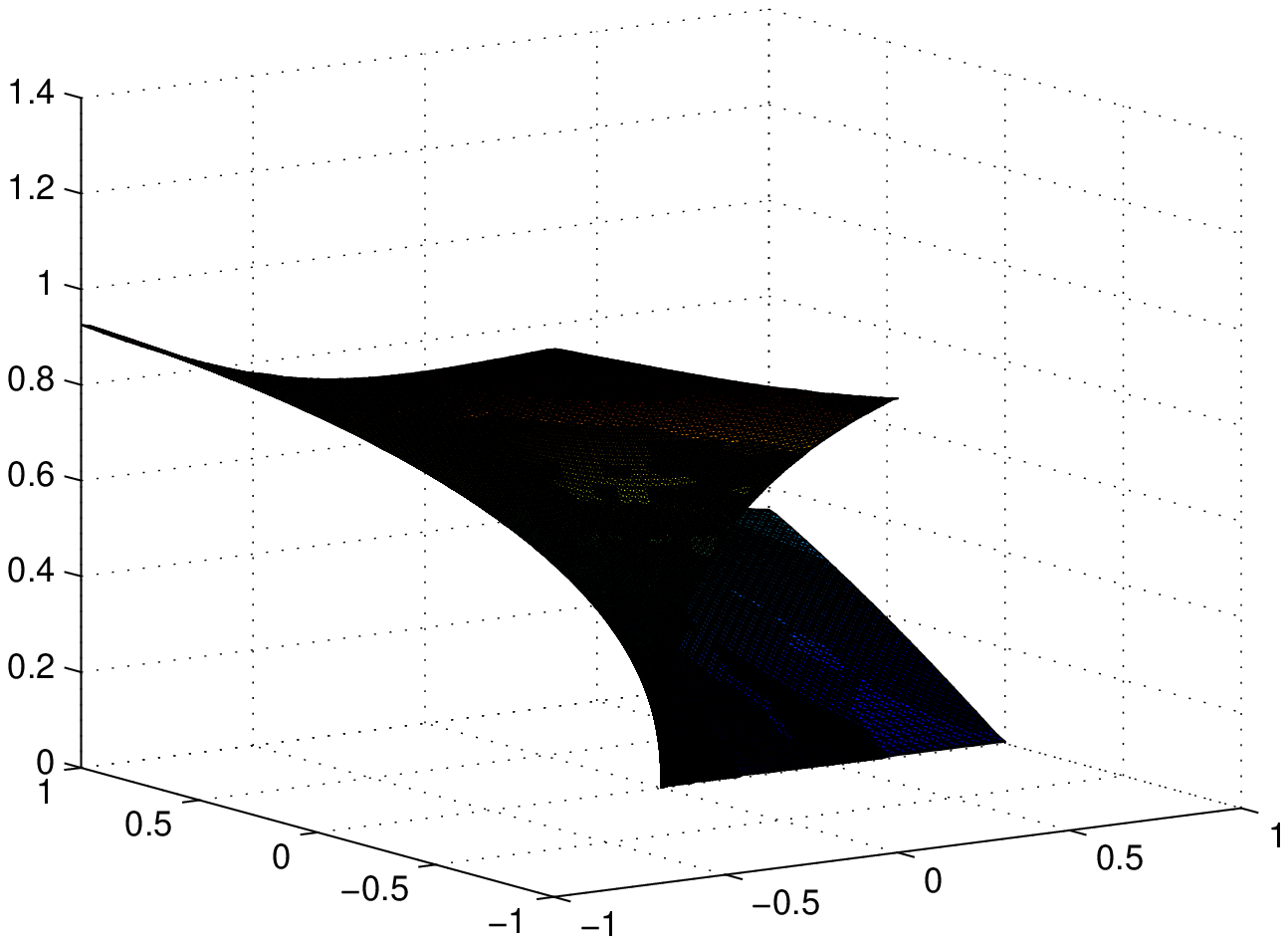}\\
  \end{minipage}
  \begin{minipage}[t]{0.5\linewidth}
    \centering
    \includegraphics[width=2.5in]{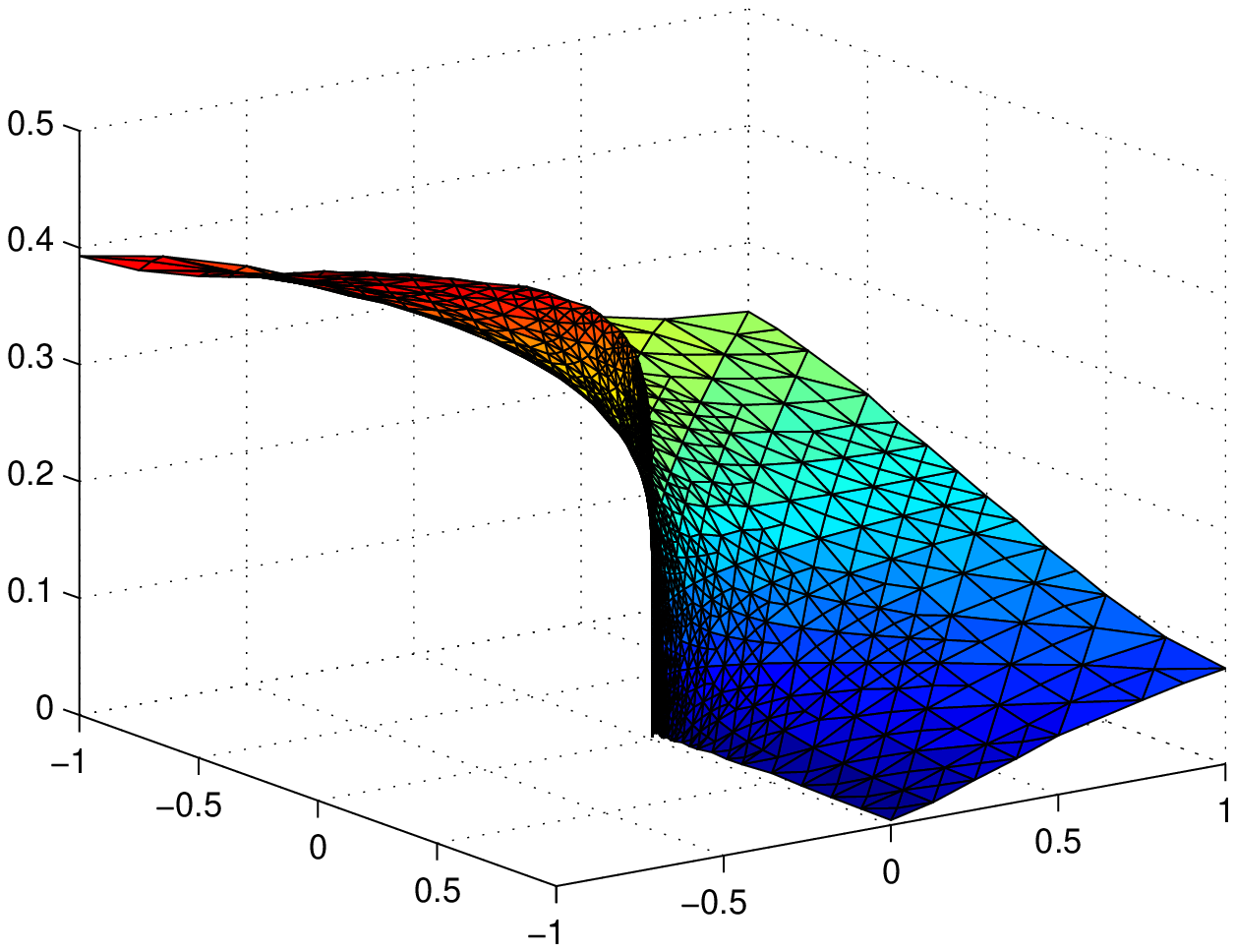}\\
  \end{minipage}
\addtocontents{lof}{figure}{FIG 7.4. {\small {\it The postprocessing
approximation to
 the pressure on the adaptively refined mesh in case $r=0.4$ (left) and
 in case $r=0.1$ (right).}}}
\end{figure}

Fig 7.4 reports a continuous piecewise-linear postprocessing approximation to
 the pressure on the adaptively refined mesh in the case $r=0.4$ (left) and
 in the  case $r=0.1$ (right) with 24 iterations. Since the approximation to the
 pressure of the MFMFE method is piecewise constant, the value of
 the  postprocessing approximation to the pressure on each node is
 taken as the algorithmic mean of the values of the pressure finite
 element solution on all the elements sharing the vertex.   
% From Fig 7.4 we
% see that the postprocessing approximation to the pressure obtains a
% satisfactory result.
\begin{figure}[htbp]
  \begin{minipage}[t]{0.5\linewidth}
    \centering
    \includegraphics[width=2.5in]{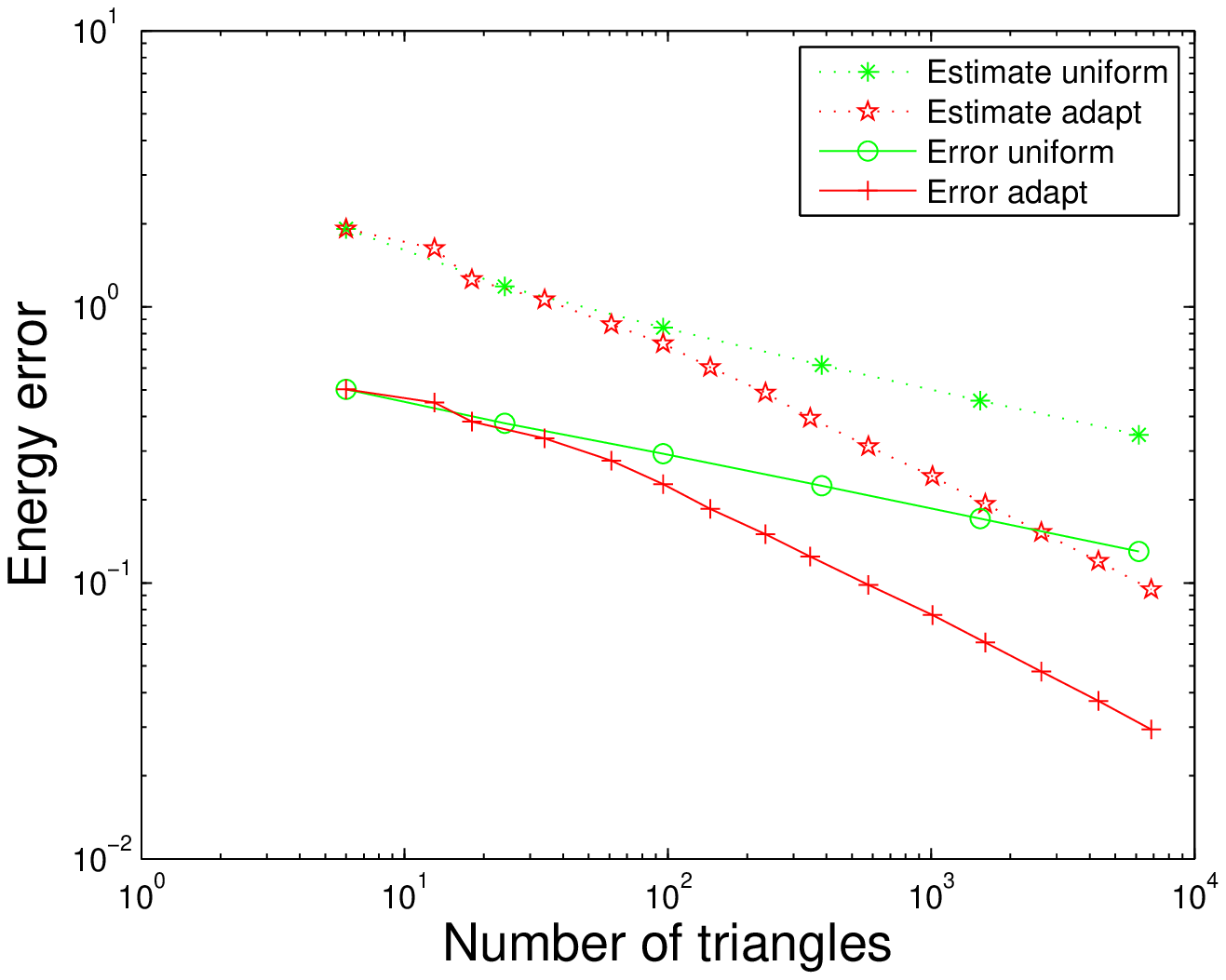}\\
  \end{minipage}
  \begin{minipage}[t]{0.5\linewidth}
    \centering
    \includegraphics[width=2.5in]{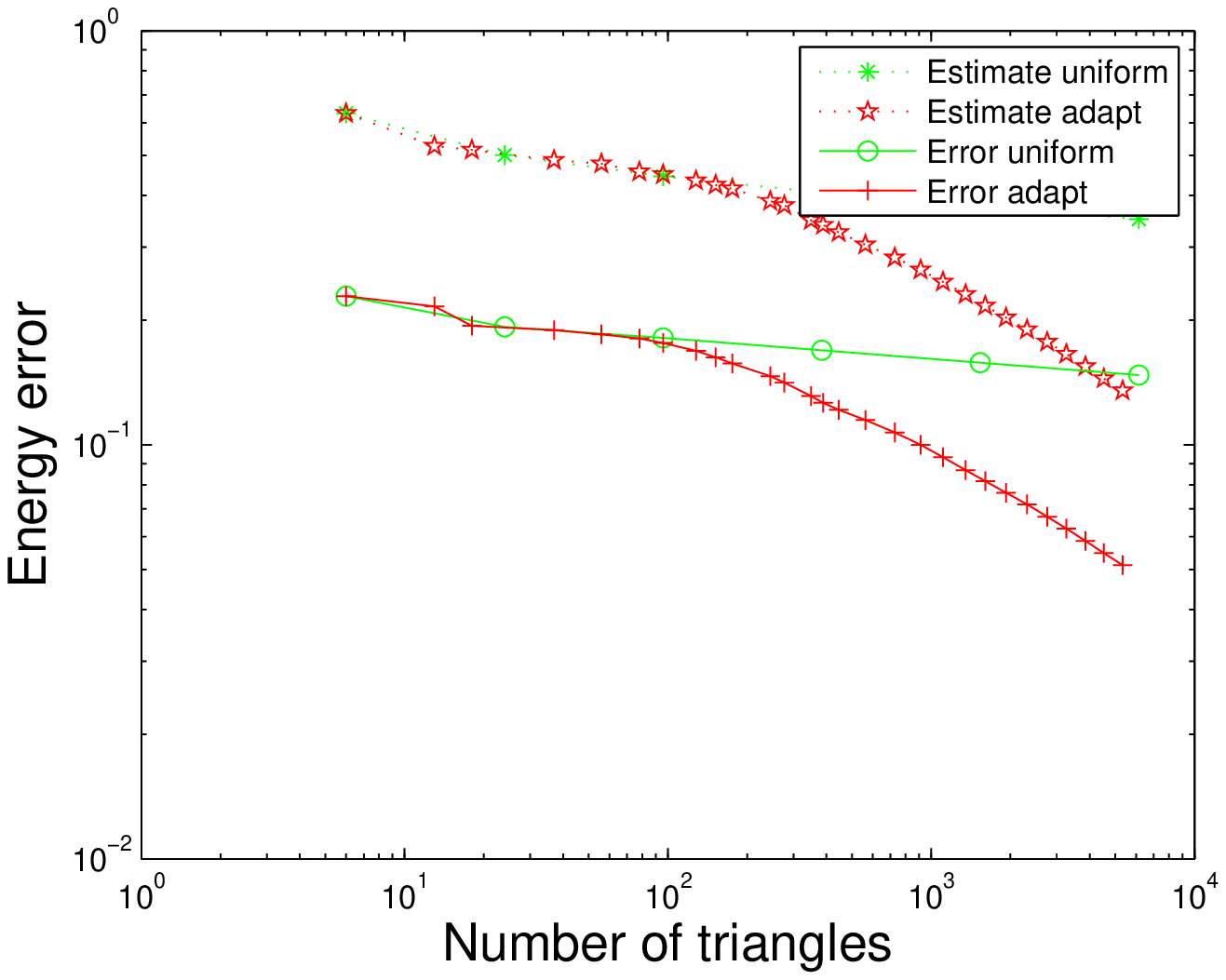}\\
  \end{minipage}
\addtocontents{lof}{figure}{FIG 7.5. {\small {\it The estimated and
actual errors against the number of elements in uniformly /
adaptively refined meshes in case $r=0.4$ (left)
 and in case $r=0.1$ (right) with the marking parameter $\tilde\theta=0.5$.}}}
\end{figure}

Fig 7.5 reports the estimated and actual errors of the numerical
solutions on uniformly and adaptively refined meshes. It can be seen
that the error of the velocity in $L^{2}$ norm uniformly reduces
with a fixed factor on two successive meshes, and that the error on
the adaptively refined meshes decreases more rapidly than the one on
the uniformly refined meshes. This means that one can substantially
reduce the number of unknowns necessary to obtain the prescribed
accuracy by using  {\it a posteriori } error estimators and adaptively
meshes. We note that the exact error is approximated with a 7-point
quadrature formula in each triangle.
\begin{figure}[htbp]
  \begin{minipage}[t]{0.5\linewidth}
    \centering
    \includegraphics[width=2.5in]{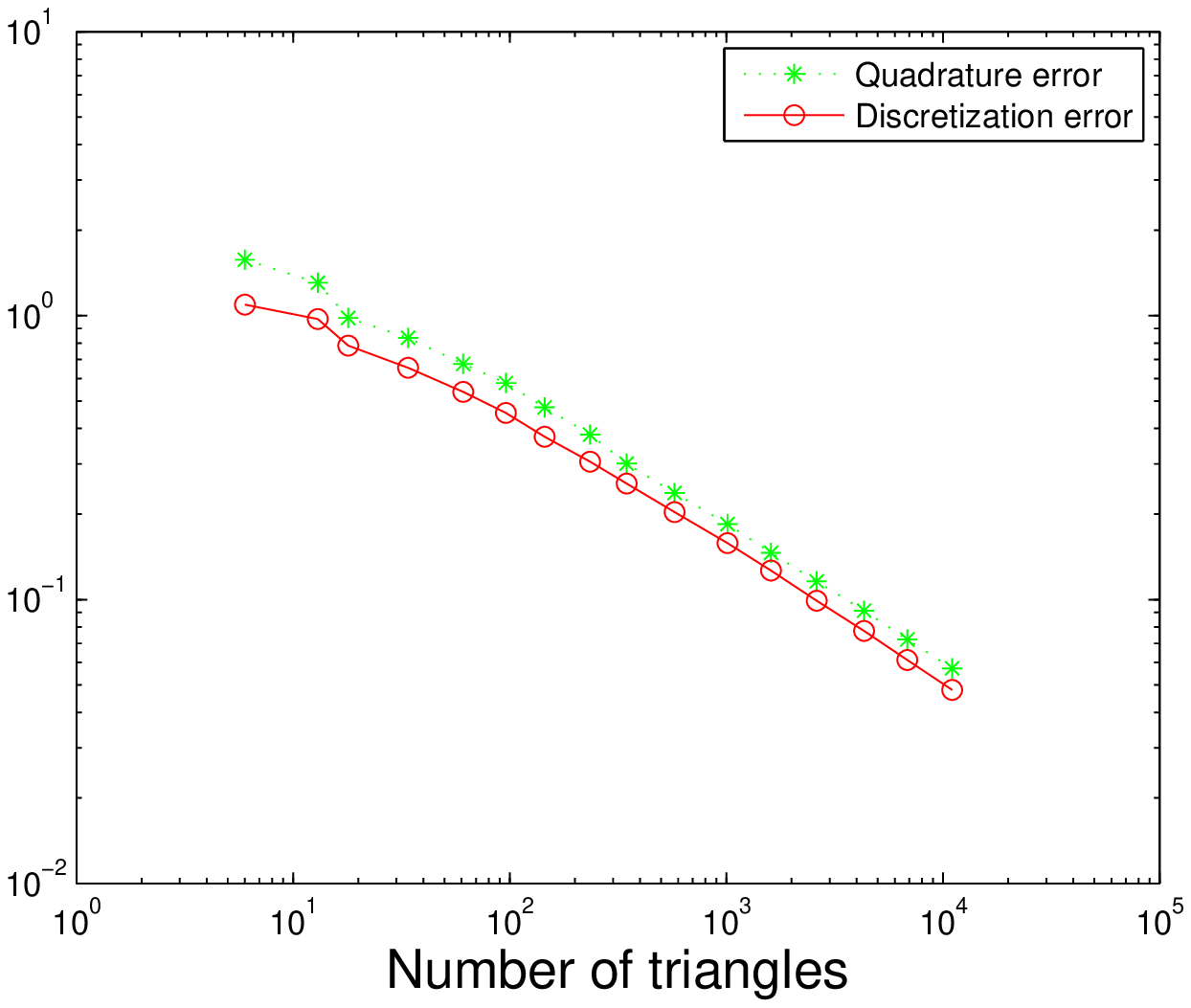}\\
  \end{minipage}
  \begin{minipage}[t]{0.5\linewidth}
    \centering
    \includegraphics[width=2.5in]{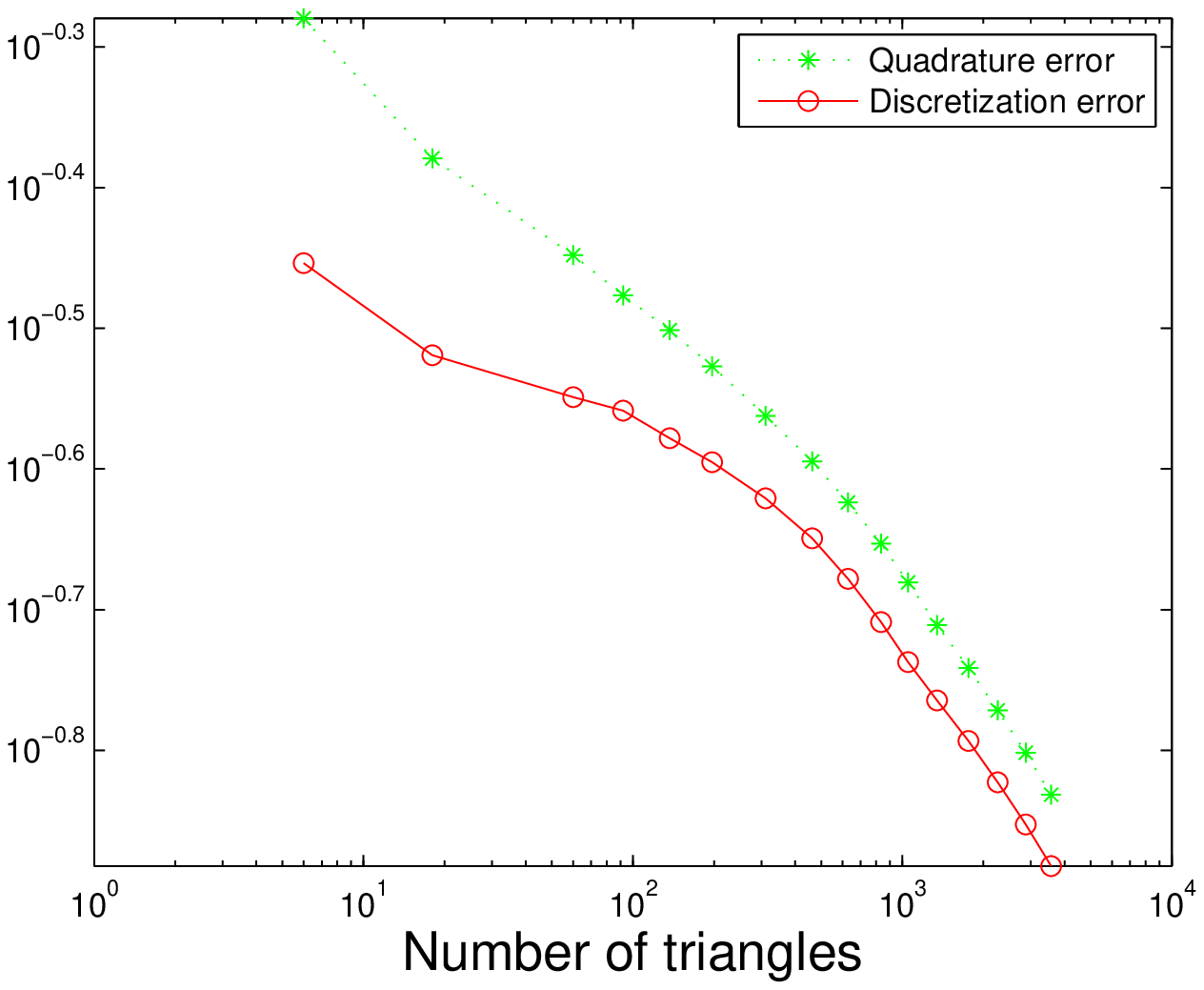}\\
  \end{minipage}
\addtocontents{lof}{figure}{FIG 7.6. {\small {\it The quadrature
error $\eta_{Q}$ and discretization error $\eta_{h}$ against the
number of elements in adaptively refined meshes in case $r=0.4$
with the marking parameter $\tilde\theta=0.5$ (left)
 and in case $r=0.1$ with the marking parameter $\tilde\theta=0.8$ (right) .}}}
\end{figure}

Fig 7.6 shows the quadrature error $\eta_{Q}$ and discretization
error $\eta_{h}$ in adaptively refined meshes in case $r=0.4$
with the marking parameter $\theta=0.5$ (left)
 and in case $r=0.1$ with the marking parameter $\theta=0.8$
 (right). It can be seen that the error indicator $\eta_{h}$
 produced by the discretization is very close to the error indicator
 $\eta_{Q}$ produced by the quadrature rule as the mesh is refined.
 This also shows that the quadrature indicator $\eta_{Q}$ is
 very efficient. We note that this efficiency is not sufficiently demonstrated by Theorem \ref{equation9} due to the appearance of the pressure error term, while this error term usually has the second order accuracy on   uniform meshes.

\subsection*{Example 7.2} We consider the problem
(\ref{equation1}) in a square domain $\Omega=(-1,1)\times(-1,1)$  with Dirichlet boundary conditions,
where $\Omega$ is divided into four subdomains $\Omega_{i}$
($i=1,2,3,4$) corresponding to the axis quadrants (in the
counterclockwise direction), and the permeability $K$ is piecewise
constant with $K=s_{i}I$ in $\Omega_{i}$. We assume the exact
solution of this model has the form
\begin{equation*}
p(\rho,\theta)|_{\Omega_{i}}=\rho^{r}(a_{i}sin(r\theta)+b_{i}cos(r\theta)).
\end{equation*}
 Here
$\rho,\theta$ are the polar coordinates in $\Omega$, $a_{i}$ and
$b_{i}$ are constants depending on $\Omega_{i}$, and $r$ is a
parameter. This solution is not continuous across the interfaces,
and only the normal component of its velocity ${\bf u}=-K\nabla p$
is continuous, and it  exhibits a strong singularity at the
origin. We consider a set of coefficients in the following table:
\begin{center}
\scriptsize
\begin{tabular}{|c|} \hline
 $s_{1}=s_{3}=5$, $s_{2}=s_{4}=1$\\ \hline $r=0.53544095$\\
\hline
$a_{1}=\ \ 0.44721360$, $b_{1}=\ \ 1.00000000$\\
$a_{2}=-0.74535599$,  $b_{2}=\ \ 2.33333333$\\
$a_{3}=-0.94411759$,  $b_{3}=\ \ 0.55555555$\\
$a_{4}=-2.40170264$, $b_{4}=-0.48148148$ \\ \hline
\end{tabular}
\end{center}

The origin mesh consists of 8 right-angled triangles. We perform the
adaptive algorithm described in Example 7.1 with the marking
parameter $\tilde\theta=0.5$.  Figs 7.7-7.8 report the
adaptive meshes generated by 6 to 8 iterations, and the continuous piecewise-linear
postprocessing approximation to the pressure on the adaptively
refined mesh. We again see that the refinement concentrates around
the origin. This indicates that the predicted error estimator
captures well the singularity of the solution.
%, and that the
%postprocessing approximation is satisfactory.
\begin{figure}[htbp]
  \begin{minipage}[t]{0.5\linewidth}
    \centering
    \includegraphics[width=2.5in]{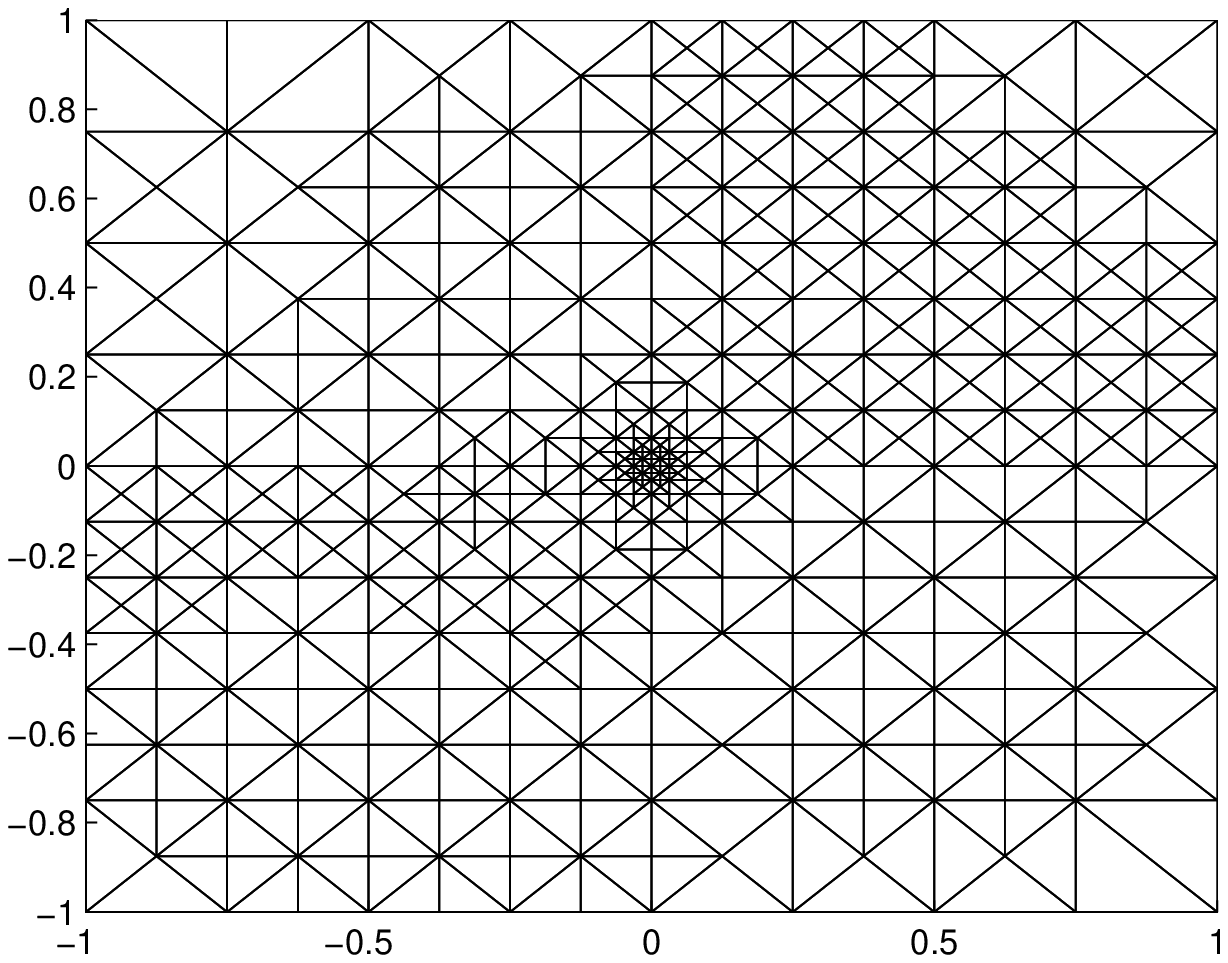}\\
  \end{minipage}
  \begin{minipage}[t]{0.5\linewidth}
    \centering
    \includegraphics[width=2.5in]{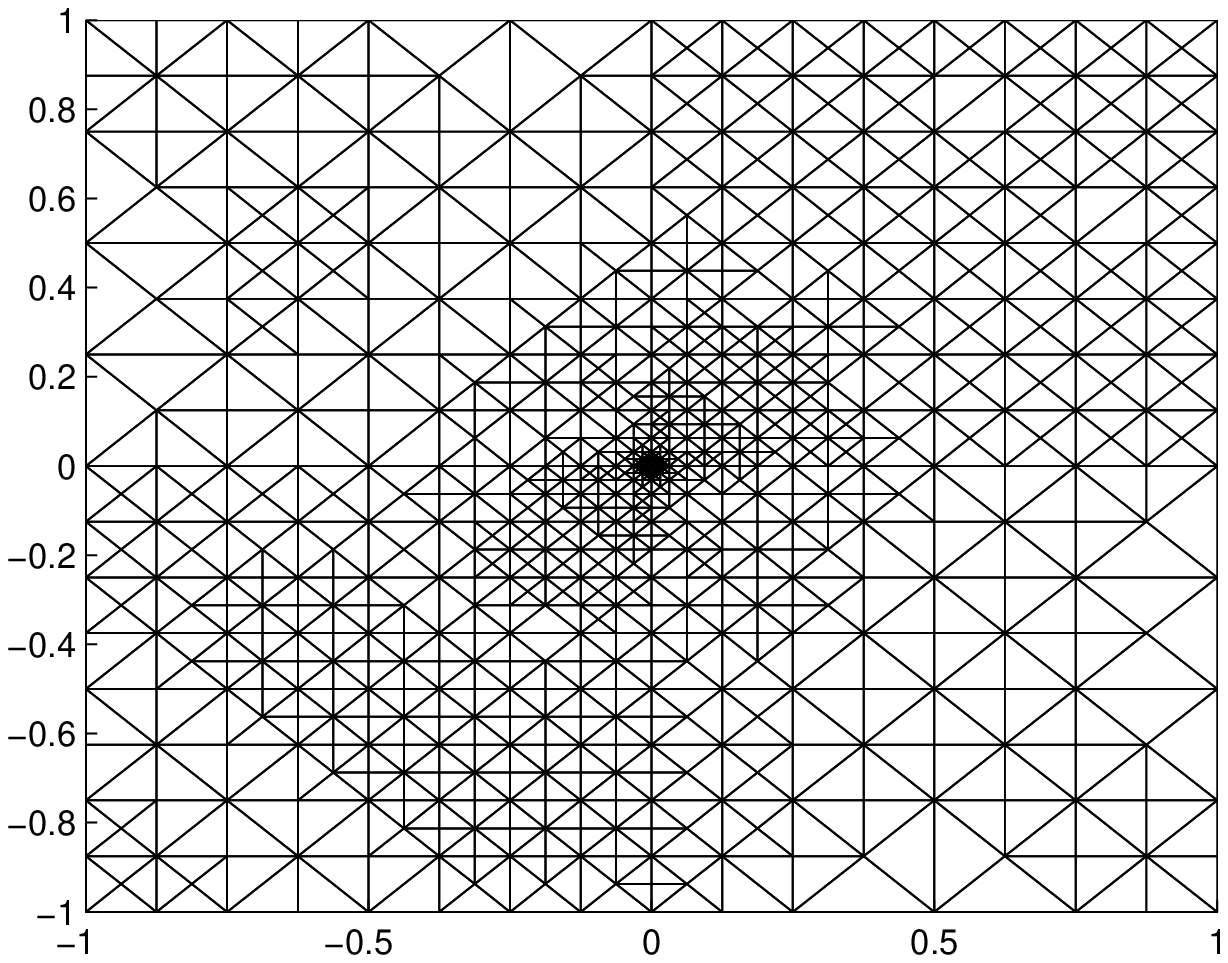}\\
  \end{minipage}
\addtocontents{lof}{figure}{FIG 7.7. {\small {\it A mesh with 740
triangles, iteration 6 (left) and a mesh with 1350 triangles,
iteration 7 (right).}}}\\
\end{figure}
\begin{figure}[htbp]
  \begin{minipage}[t]{0.5\linewidth}
    \centering
    \includegraphics[width=2.5in]{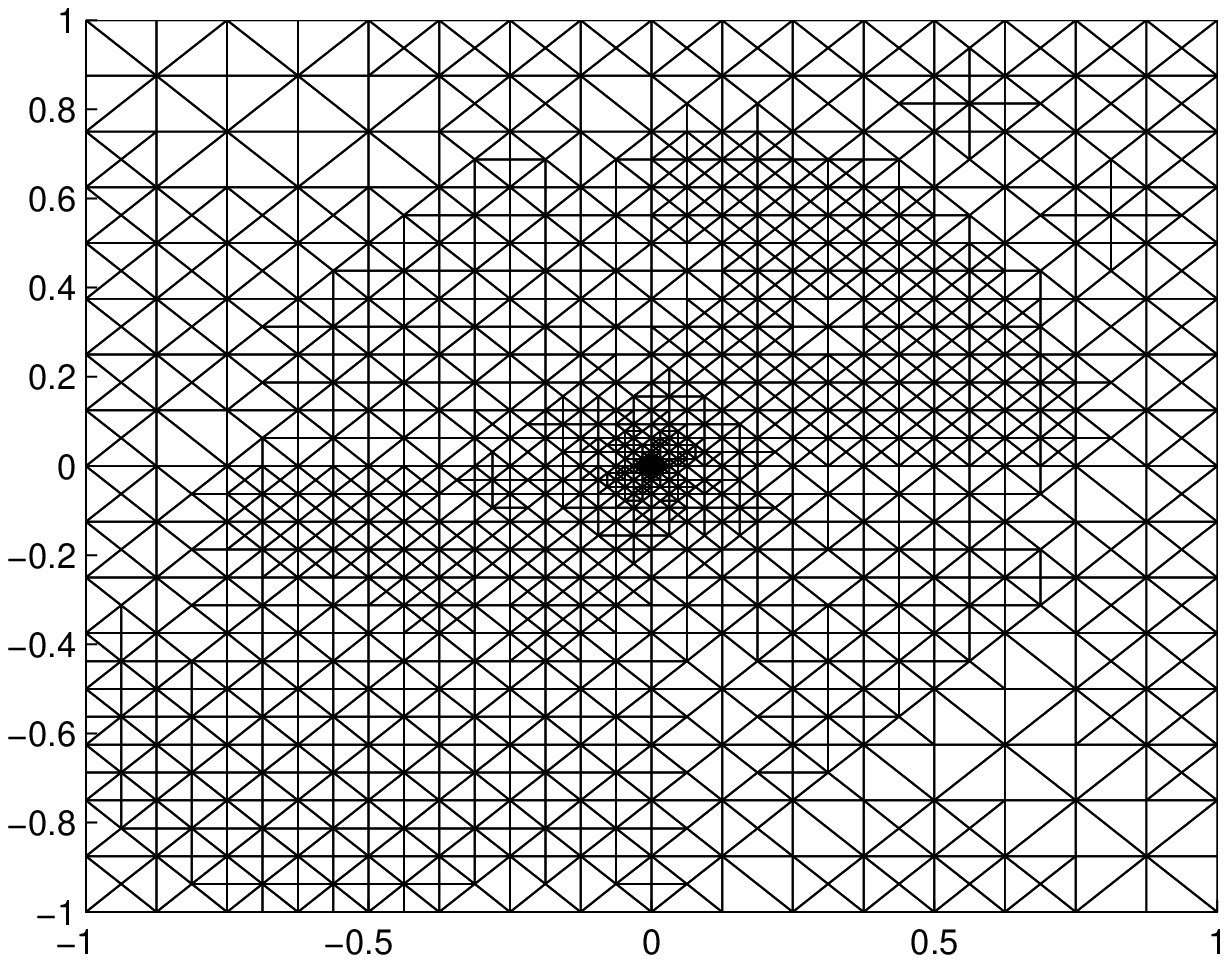}\\
  \end{minipage}
  \begin{minipage}[t]{0.5\linewidth}
    \centering
    \includegraphics[width=2.5in]{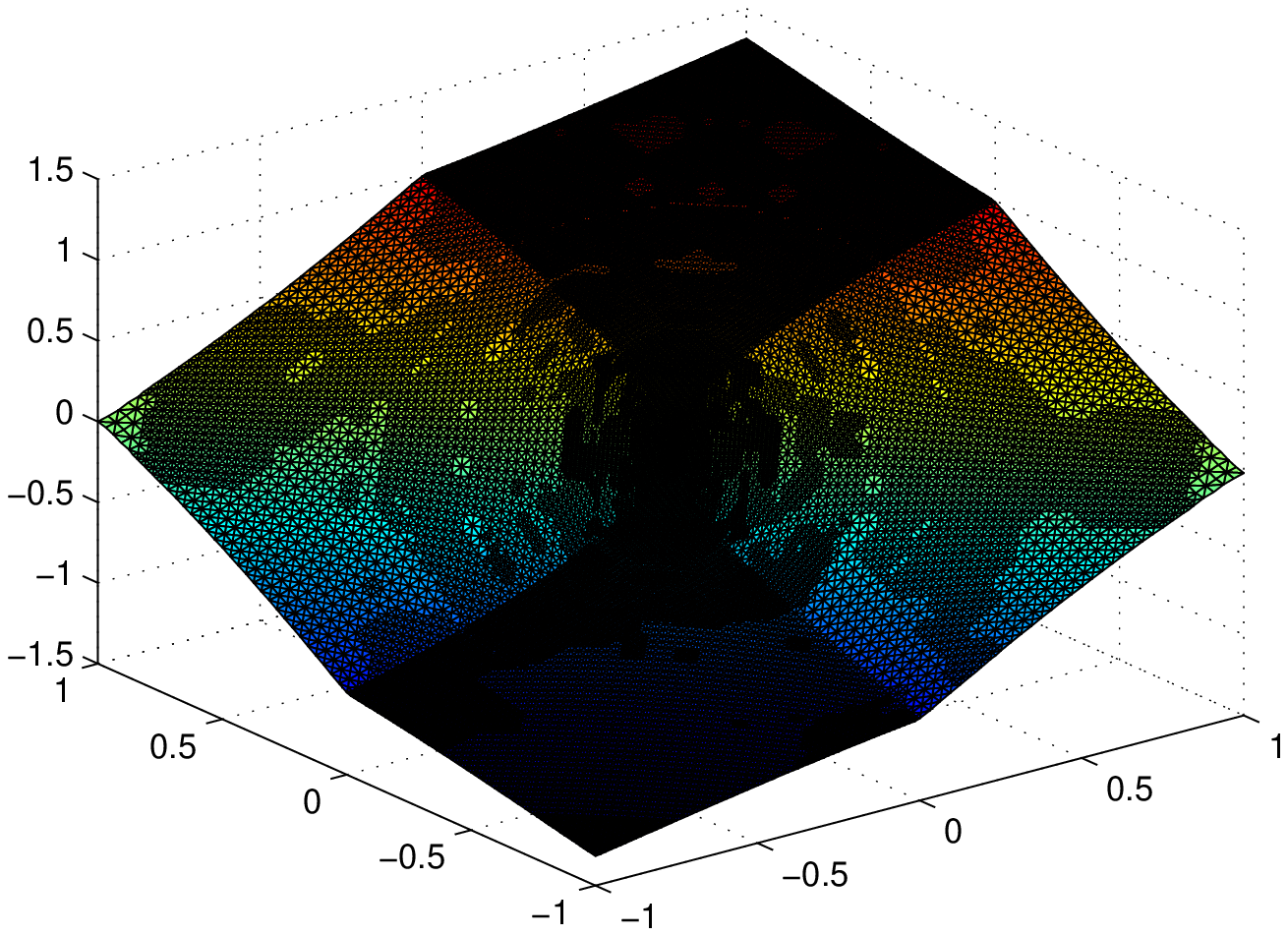}\\
  \end{minipage}
\addtocontents{lof}{figure}{FIG 7.8. {\small {\it  A mesh with 2328
triangles, iteration 8 (left) and the postprocessing
approximation to the pressure on the adaptively refined mesh.}}}
\end{figure}

Fig 7.9 reports the estimated and actual errors of the numerical
solutions on uniformly and adaptively refined meshes (left), and the
quadrature indicator $\eta_{Q}$ and discretization indicator
$\eta_{h}$ in adaptively refined meshes (right).
\begin{figure}[htbp]
  \begin{minipage}[t]{0.5\linewidth}
    \centering
    \includegraphics[width=2.5in]{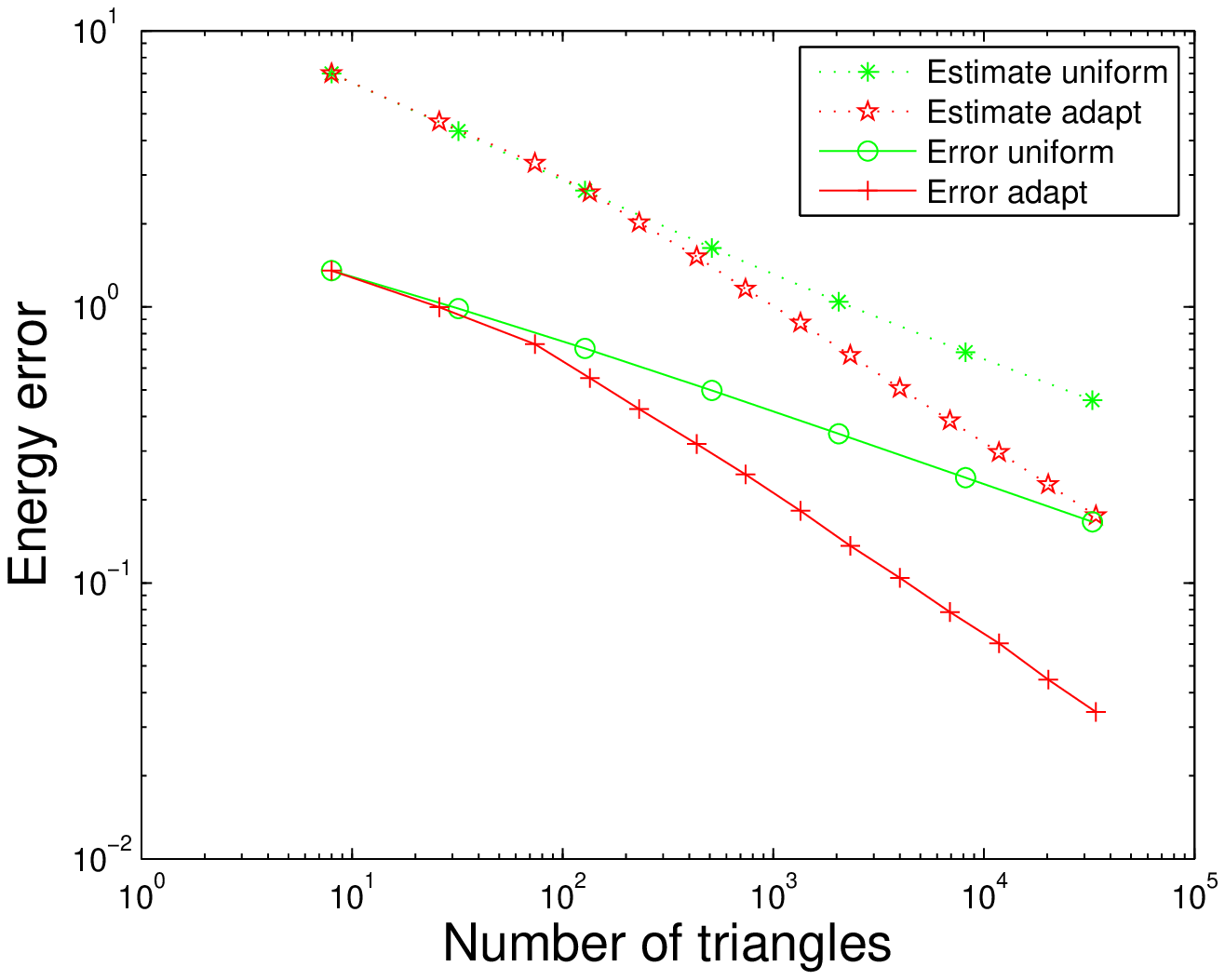}\\
  \end{minipage}
  \begin{minipage}[t]{0.5\linewidth}
    \centering
    \includegraphics[width=2.5in]{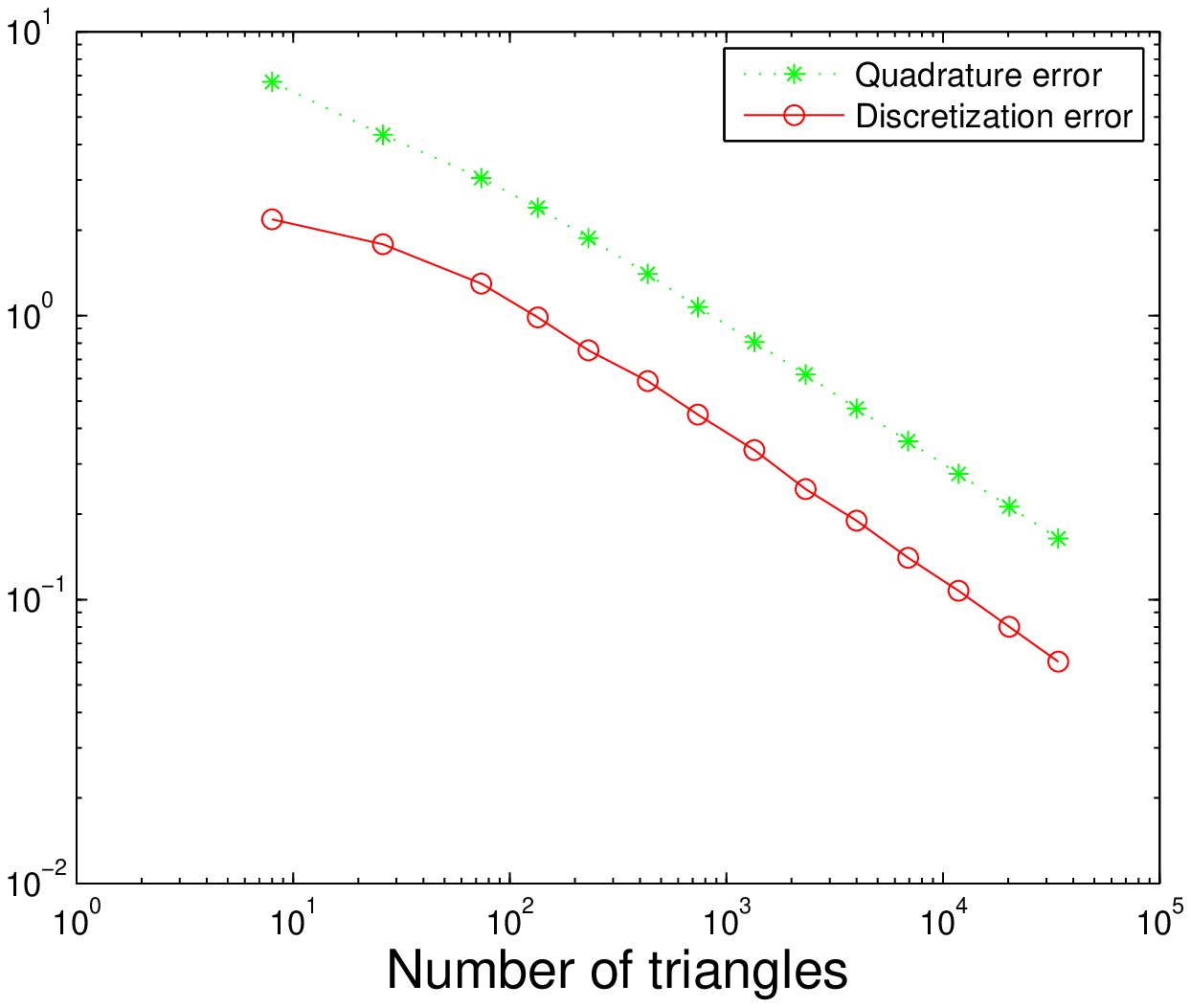}\\
  \end{minipage}
\addtocontents{lof}{figure}{FIG 7.9. {\small {\it The estimated and
actual errors against the number of elements in uniformly /
adaptively refined meshes (left) and the quadrature error $\eta_{Q}$
and discretization error $\eta_{h}$ against the number of elements
in adaptively refined meshes (right).}}}
\end{figure}

We can see that the error of the velocity uniformly
reduces with a fixed factor on two successive meshes,   that the
error on the adaptively refined meshes decreases more rapidly than
the one on the uniformly refined meshes, and that the a posteriori
error estimators developed in this paper are efficient with respect
to inhomogeneities and anisotropy of the permeability. This means
that one can substantially reduce the number of unknowns necessary
to obtain the prescribed accuracy by using {\it a posteriori } error
estimators and adaptively refined meshes.  We also
see that the error indicator $\eta_{h}$ and $\eta_{Q}$ differs
at most a constant factor, which shows   the quadrature error
estimator $\eta_{Q}$ is   efficient.

\section{Conclusions} In this contribution we have developed a reliable and efficient   a posteriori error estimator of residual-type  for the multi-point flux mixed finite element methods for flow in porous media in two or three space dimensions.  The main tools of  our analysis are  a locally postprocessed technique and a quadrature error estimation. Numerical experiments are conformable to our theoretical  results.

\end{document}